\documentclass[preprint]{elsarticle}

\usepackage{graphicx}
\usepackage{amsmath}
\usepackage{amssymb}
\usepackage{enumerate}
\usepackage{float}
\usepackage{color}

\usepackage{amsthm}
\usepackage[T1]{fontenc}
\usepackage[para]{threeparttable}
\usepackage[margin=1in]{geometry}
\usepackage{booktabs}
\usepackage{multirow}
\usepackage{longtable}
\usepackage{booktabs}
\usepackage{tabularx}
\usepackage{lscape}
\usepackage{threeparttable}
\usepackage{tablefootnote}
\usepackage{tikz}
\usepackage{amsmath}
\usepackage{algorithm}
\usepackage[noend]{algpseudocode}
\usetikzlibrary{positioning}
\usepackage{hyperref}
\makeatletter
\usepackage{multirow}
\usepackage{lscape}
\usepackage{longtable}
\usepackage{lineno,hyperref}
\usepackage{setspace}
\modulolinenumbers[5]
\usepackage{endnotes}
\let\footnote=\endnote

\usepackage{natbib}
 \bibpunct[, ]{(}{)}{,}{a}{}{,}%
 \def\bibsep{\smallskipamount}%

\usepackage{array}

\usepackage{pdflscape}

\def\reals{\mathbb{R}}

\def\comp{\raise 1pt \hbox{$\scriptstyle\circ$}}

\def\argmax{\mathop{\rm argmax}\limits}

\def\min{\mathop{\rm min\ }\limits}
\def\max{\mathop{\rm max\ }\limits}
\def\st{\mathop{\rm subject\ to\ }}

\def\upto{{\raise 1pt \hbox{$\scriptstyle \,\nearrow\,$}}}
\def\downto{{\raise 1pt \hbox{$\scriptstyle \,\searrow\,$}}}

\newtheorem{theorem}{Theorem}

\biboptions{authoryear}	

\begin{document}
\begin{frontmatter}
	\title{A General Purpose Exact Solution Method for Mixed Integer Concave Minimization Problems}
	\author{Ankur Sinha}
\ead{asinha@iima.ac.in}
\author{Arka Das}
\ead{phd17arkadas@iima.ac.in}
\author{Guneshwar Anand}
\ead{guneshwar.anand-phd19@iimv.ac.in}
\author{Sachin Jayaswal}
\ead{sachin@iima.ac.in}
\address{Production and Quantitative Methods, Indian Institute of Management Ahmedabad, Gujarat, India, 380015}

\begin{abstract}
		In this article, we discuss an exact algorithm for solving mixed integer concave minimization problems. A piecewise inner-approximation of the concave function is achieved using an auxiliary linear program that leads to a bilevel program, which provides a lower bound to the original problem. The bilevel program is reduced to a single level formulation with the help of Karush-Kuhn-Tucker (KKT) conditions. Incorporating the KKT conditions lead to complementary slackness conditions that are linearized using BigM, for which we identify a tight value for general problems. Multiple bilevel programs, when solved over iterations, guarantee convergence to the exact optimum of the original problem. Though the algorithm is general and can be applied to any optimization problem with concave function(s), in this paper, we solve two common classes of operations and supply chain problems; namely, the concave knapsack problem, and the concave production-transportation problem. The computational experiments indicate that our proposed approach outperforms the customized methods that have been used in the literature to solve the two classes of problems by an order of magnitude in most of the test cases. 
		%For concave production-transportation problem we consider two sub-classes: multiple sourcing and single sourcing. 
		%Our convergence analysis suggests that the proposed algorithm terminates at the global optimum.
		%In this paper, we show how the algorithm can be applied to optimization problems that contain concavities in the objective function or constraints. 
\end{abstract}
	
	
%\maketitle
\end{frontmatter}
%\linenumbers
%=================================================
\section{Introduction}\label{Intro}
%=================================================
The interest in non-convex optimization is motivated by its applications to a wide variety of real-world problems, including
%fixed charge problems \citep{murty1968solving} ({ where is the non-linearity here}), 
concave knapsack problems \citep{sun2005exact,han2017online}, production-transportation problem \citep{holmberg1999production,kuno2000lagrangian,saif2016supply}, facility location problems with concave costs \citep{soland1974optimal,ni2021branch}, location-inventory network design problems \citep{li2021location,farahani2015location,shen2007incorporating,jeet2009logistics}, concave minimum cost network flow problems \citep{guisewite1990minimum,fontes2007heuristic}, etc. %These classes of problems commonly appear in transportation, communication, production, inventory planning, resource allocation, and network design to name a few areas. 
Non-convexities often arise, in the above problems, due to the presence of concave functions either in the objective or in the constraints. It is difficult to solve these optimization problems exactly, and hence the problem has been of interest to the optimization community since the 1960s.\par
One of the earliest studies on non-convex optimization problems is by \cite{tuy1964concave}, where the author proposed a cutting plane algorithm for solving concave minimization problems over a polyhedron. The proposed algorithm was based on the partitioning of the feasible region, where the partitions were successively eliminated using Tuy cuts. However, the algorithm had a drawback since it was not guaranteed to be finite, which was tackled later by \cite{zwart1974global, majthay1974quasi}. Apart from cutting plane approaches, researchers also used relaxation-based ideas. For instance, \cite{falk1976successive,carrillo1977relaxation} computed successive underestimations of the concave function to solve the problem optimally. Branch-and-bound based approaches are also common to solve these problems, where the feasible region is partitioned into smaller parts using branching \citep{falk1969algorithm,horst1976algorithm,ryoo1996branch,tawarmalani2004global}. Other ideas for handling concavities are based on extreme point ranking \citep{murty1968solving,taha1973concave} or generalized Benders decomposition \citep{floudas1989global,li2011nonconvex}. The limitations of some of the above studies are one or more of the following: applicable only to a specific class of concave minimization problems; strong regularity assumptions; non-finite convergence; and/or convergence to a local optimum. Due to these limitations, there is also a plethora of specialized heuristics and meta-heuristics in the literature to obtain good quality or approximate solutions in {\it less} time. However, most of the heuristics and meta-heuristics do not guarantee convergence. Therefore, the idea of obtaining good quality solutions is often questioned as there is no way to ascertain how far the solution is from the global optimum.\par
In this paper, we discuss an algorithm for concave minimization problems, which requires few assumptions about the problem structure thereby making our approach general. We convert the concave minimization problem into a bilevel program that provides an approximate representation of the original problem. Solving bilevel programs over multiple iterations guarantee convergence to the global optimum of the original problem. While solving a bilevel program, one often attempts to reduce the program to a single level. Interestingly, in our approach converting a single level program to a bilevel program is the key to solving the original concave minimization problem.
%We address some of the concerns related to minimization problems with concave functions by providing a new algorithm that requires few assumptions about the problem. 
The problem studied in this paper is also studied in the area of difference-of-convex (DC) programming, for instance, the work by \cite{strekalovsky2015local} comes close to our study. However, there have been challenges in directly implementing many of the DC programming approaches on operations and supply chain problems, as the problems considered are often large dimensional mixed integer problems for which obtaining the exact solution in reasonable time is difficult. We design and implement a piecewise-linear inner approximation method that is able to solve large dimensional operations and supply chain problems that involve mixed integers and concavities. %The method does not require any user intervention for bigM adjustments. 
%To the best of our knowledge, the study closest to our work is on difference-of-convex (DC) programming \citep{strekalovsky2015local}, which uses inner cuts to solve a DC program. Our proposed method for concave minimization problems is general to the extent that it does not demand the concave functions or the other functions in the problem to be differentiable. 
%We also address the issues related to BigM that arise at the intermediate steps of the algorithm. 
The method relies on the piecewise-linear inner approximation \citep{rockafellar1970convex} approach, which replaces the concave function to arrive at a bilevel formulation that leads to the lower bound of the original problem. The bilevel optimization problem is solvable using the Karush-Kuhn-Tucker (KKT) approach. Through an iterative procedure, wherein multiple bilevel programs are solved, the method converges to the global optimum of the original problem. The method can be used to solve concave minimization problems with continuous or discrete variables exactly, as long as the concavities in the optimization problem are known. The reduction of the bilevel program to a single level introduces a BigM, for which we identify a tight value. Though the identification of a tight value of BigM is often downplayed as a contribution, we would like to highlight it as a major contribution, given that it is applicable to general problems, and also because it has been recently shown by \cite{kleinert2020there} that identifying a tight BigM in similar contexts is an NP hard problem. The choice of an appropriate value for BigM, in our case, makes the method highly competitive and applicable to large dimensional mixed integer problems. 
%Though the identification of a tight value of BigM is often downplayed as a contribution, we would like to highlight it as a major contribution since in the context of bilvel optimization, it has been recently proven to be an NP hard problem \citep{kleinert2020there}. 
The structure of the bilevel problem is exploited to arrive at a tight value of BigM for general problems. We solve two classes of optimization problems in this paper to demonstrate the efficacy of our method: (i) concave knapsack problem; and (ii) concave production-transportation problem. For the concave production-transportation problem, we further consider two sub-classes: (a) single sourcing; and (b) multiple sourcing that have quite different formulations. We show that the proposed exact method, which is general, beats the existing specialized methods for solving the application problems by a large margin.\par
%
%proposed to handle such problems; on the other hand we also compare our method with a fairly recent exact method that is based on ideas from the domain of DC (difference of convex function) programming \citep{strekalovsky2015local,horst1999dc}. 
The rest of the paper is organized as follows. We provide the algorithm description, followed by the convergence theorems and proofs in Section~\ref{sec:algorithmDescription}. The concave knapsack problems and production-transportation problems are discussed in Section~\ref{sec:knapsackprob} and Section~\ref{sec:prodtransprob}, respectively. Each of these sections contains a brief survey, problem description, and computational results for its respective problem. Finally, we conclude in Section~\ref{sec:conclusions}. The paper also has an Appendix, where we show the working of the algorithm on two sample problems.
%, one with concavity in the objective (see Appendix~\ref{sec:NumEx}) and the other with concavity in the constraints (see Appendix~\ref{sec:constraints}).

\section{Algorithm Description}\label{sec:algorithmDescription}
%Do not delete this:
%We consider optimization problems of the following kind
%\begin{align}
%\minimize_{x} & \; f(x) \label{eq:objOrig}\\
%\st & \\
% & h(x) + \phi(x) \leq 0 \label{eq:consConcaveOrig}\\
% & g_i(x) \leq 0, \quad i=1,\dots,I \label{eq:consConvexOrig}
%\end{align}
%where $f(x)$, $g(x)$ and $h(x)$ are convex and $\phi(x)$ is concave. Note that if one substitutes $\phi'(x) = -\phi(x)$, then $\phi'(x)$ is convex and \eqref{eq:consConcaveOrig} can be written as $h(x) - \phi'(x)$, which is a DC (difference of convex functions) constraint. In case concavity is present in the objective function as shown below:
%\begin{align}
%\minimize_{x} & \; f(x)+\phi(x) \label{eq:objConcave}\\
%\st & \\
% & g_i(x) \leq 0, \quad i=1,\dots,I \label{eq:consConvex}
%\end{align}
%then the above formulation can be written in the form of \eqref{eq:objOrig}-\eqref{eq:consConvexOrig}, by substituting a variable $t$ as follows:
%\begin{align}
%\minimize_{x} & \; f(x)+t\\
%\st & \\
% & \phi(x) \le t\\
% & g_i(x) \leq 0, \quad i=1,\dots,I
%\end{align}
We consider optimization problems of the following kind
\begin{align}
	\min_{x} & f(x) + \phi(x) \label{eq:objOrig}\\
	\st & g_i(x) \leq 0, \qquad\qquad i=1,\dots,I \\ 
	& x^l_k \le x_k \le x^u_k, \qquad k = 1,\ldots,n \label{eq:consConvexOrig}
\end{align}
where $f(x)$ and $g(x)$ are convex, $\phi(x)$ is strictly concave and $x \in \reals^{n}$. Note that there is no restriction on $x$, which may be integer or continuous. The functions are assumed to be Lipschitz continuous. 
%For a given set of points $S_c = \{ z^{1}, z^{2}, \ldots, z^{\tau} \}$ (let $c=1$), the function $\phi(x)$ can be approximated as follows (\cite{rockafellar1970convex}):
For a given set of $\tau$ sample points $S_c = \{ z^{1}, z^{2}, \ldots, z^{\tau} \}$, where each $z^{j}$ is an $n$-dimensional point, the function $\phi(x)$ can be approximated as follows: (\cite{rockafellar1970convex}):
%see above line
\begin{align}
\hat{\phi}(x|S_c) = \max_{\mu} \left\{ \sum_{j=1}^{\tau} \mu_j \phi(z^{j}): \sum_{j=1}^{\tau} \mu_j=1, \sum_{j=1}^{\tau} \mu_j z^{j}_k=x_k, k=1,\ldots,n, \mu_j\ge0, j=1,\ldots,\tau \right\} \label{eq:innerApprox}
\end{align}
which is a linear program with $x$ as a parameter and $\mu$ as a decision vector for the linear program. 
%The algorithm for solving the concave minimization problem and the convergence proofs have been discussed using the above piecewise approximation function. However, the above approximation has minor drawbacks that we will discuss later and suggest an alternative piecewise approximation formulation. 
For brevity, we will represent the approximation $\hat{\phi}(x|S_c)$ as $\hat{\phi}(x)$.

Figures~\ref{fig:expl1} and~\ref{fig:expl2} provide an illustration of the idea behind the working of the above linear program for piecewise approximation of a concave function in a single variable ($n=1$). Figure~\ref{fig:expl1} shows how a linear combination of the approximation points are used to represent the entire shaded region. Thereafter, Figure~\ref{fig:expl2} shows how maximization in the shaded region leads to $\hat{\phi}(x)$ for any given $x$. Figure~\ref{fig:approximation4points} again represents the piecewise approximation graphically and shows how it improves with the addition of a new point in the approximation set. Note that the feasible region with 5 points is a superset of the feasible region with 4 points. We will use this property later while discussing the convergence properties of the algorithm.
	\begin{figure}
		\centering
		\begin{minipage}{.49\textwidth}
			\centering
			\includegraphics[width=1.0\linewidth]{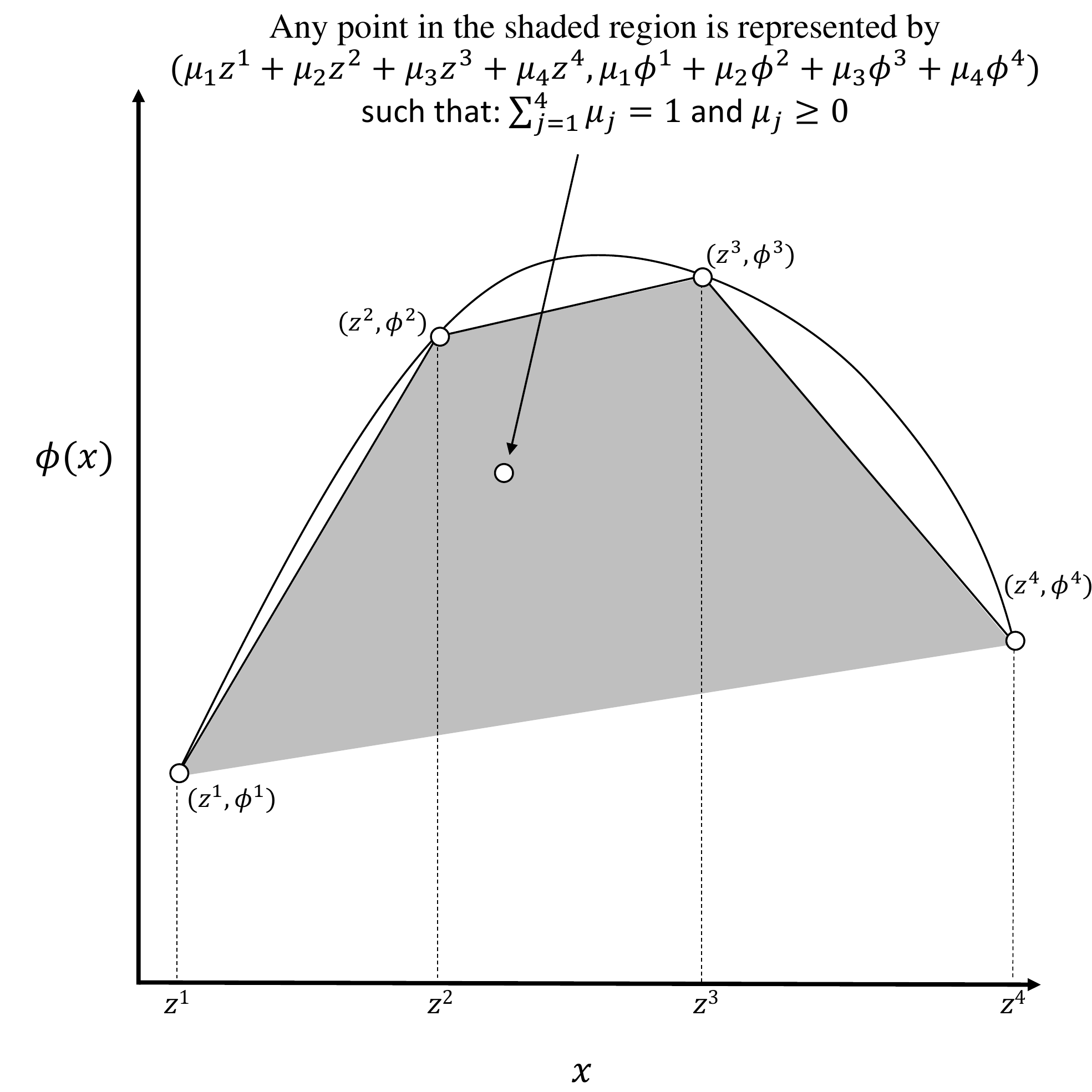}
			\caption{Linear combination of the corner points representing the entire shaded region.}
			\label{fig:expl1}
		\end{minipage}\hfill
		\begin{minipage}{.49\textwidth}
			\centering
			\includegraphics[width=1.0\linewidth]{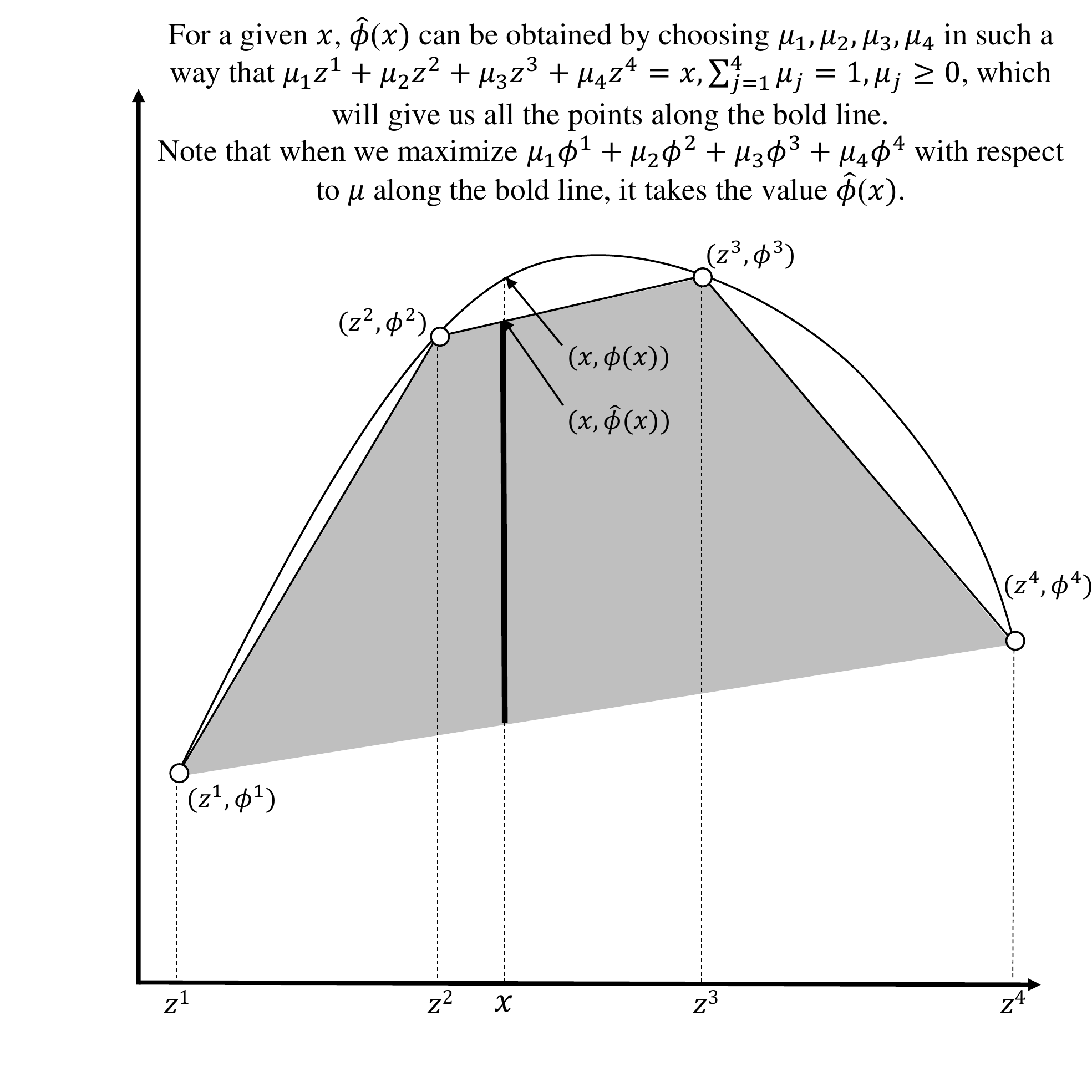}
			\caption{For a given $x$, maximization along the bold line leads to the value of $\hat{\phi}(x)$.}
			\label{fig:expl2}
		\end{minipage}
	\end{figure}

\begin{figure}
		\centering
			\includegraphics[width=.6\linewidth]{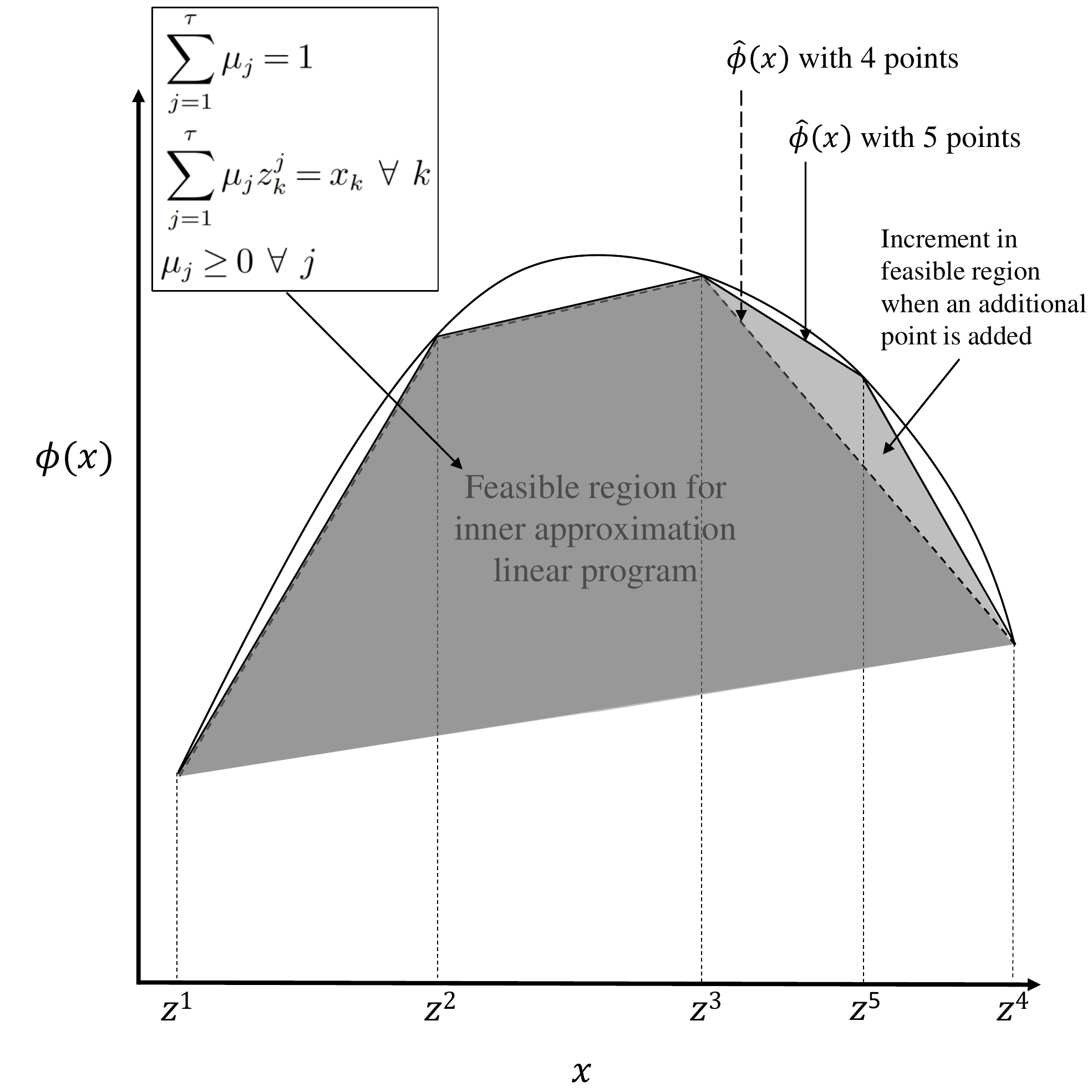}
			\caption{Inner-approximation of $\phi(x)$ with a set of points ($\tau=4$ and $\tau=5$).}
			\label{fig:approximation4points}
\end{figure}
The above approximation converts the concave minimization problem \eqref{eq:objOrig}-\eqref{eq:consConvexOrig} into the following lower bound program, as $\hat{\phi}(x)$ is the inner piecewise linear approximation of $\phi(x)$.
\begin{align}
%	\min_{x} & \left\{ f(x) + \max_{\mu} \left\{ \sum_{j=1}^{\tau} \mu_j \phi(x^{j}): \sum_{j=1}^{\tau} \mu_j=1, \sum_{j=1}^{\tau} \mu_j x^{j}_k=x_k, k=1,\ldots,n, \mu_j\ge0, j=1,\ldots,\tau \right\} \right\} \label{eq:objRel}\\
\min_{x} & f(x) + \hat{\phi}(x) \label{eq:objRel}\\
		\st & g_i(x) \leq 0, \qquad\qquad i=1,\dots,I \\ 
	& x^l_k \le x_k \le x^u_k, \qquad k = 1,\ldots,n \label{eq:consRel}
\end{align}
\begin{theorem}\label{th:lowerBound}
The value of the optimal solution of the formulation \eqref{eq:objRel}-\eqref{eq:consRel} provides a lower bound for the formulation \eqref{eq:objOrig}-\eqref{eq:consConvexOrig}.
\end{theorem}
\begin{proof}
Given that $\hat{\phi}(x)$ is a piecewise inner-approximation of $\phi(x)$, the function $\hat{\phi}(x)$ always bounds $\phi(x)$ from below. Therefore, at any given $x$, $\hat{\phi}(x)$ will always take a smaller value than $\phi(x)$. This implies the following:
$$
f(x) + \hat{\phi}(x) \le f(x) + \phi(x)
$$
\end{proof}
\noindent Formulation \eqref{eq:objRel}-\eqref{eq:consRel} is a bilevel program, which can be written as follows:
\begin{align}
	\min_{x,\zeta} & f(x) + \zeta \\
	\st & \mu \in \argmax_{\mu} \left\{ \sum_{j=1}^{\tau} \mu_j \phi(z^{j}): \sum_{j=1}^{\tau} \mu_j=1, \sum_{j=1}^{\tau} \mu_j z^{j}_k=x_k, k=1,\ldots,n, \mu_j\ge0, j=1,\ldots,\tau \right\} \label{eq:bilevelproblem}\\
	& \sum_{j=1}^{\tau} \mu_j \phi(z^{j}) \le \zeta\\
 & g_i(x)\leq 0, \quad i=1,\dots,I\\
 & x^l_k \le x_k \le x^u_k, \quad k=1,\ldots,n
\end{align}

%Feasible Region for Inner Approximation
%\begin{align*}
% & \sum_{j=1}^{\tau} \mu_j=1\\
% & \sum_{j=1}^{\tau} \mu_j z^{j}_k=x_k \;\; \forall \;\; k\\
% & \mu_j\ge 0 \;\; \forall \;\; j
%\end{align*}
A bilevel problem, where the lower level is a linear program, is often solved by replacing the lower level with its KKT conditions. Substituting the KKT conditions for the lower level program using $\alpha$, $\beta_k$ and $\gamma_j$ as the Lagrange multipliers for the constraints $\sum_{j=1}^{\tau} \mu_{j}=1$, $\sum_{j=1}^{\tau} \mu_{j} z^{j}_k=x_k$ and $\mu_{j} \geq 0$, respectively, the above formulation reduces to the following.
%Write KKT conditions here with Lagrange multipliers as $\gamma$ (for equality) and $\lambda$ (for inequalities).
\begin{align}
\textbf{Mod-$S_c$} \quad \quad \quad \quad
\min_{x,\alpha,\beta,\gamma,\zeta} & \; f(x) + \zeta \label{eq:relaxedFormulationObj}\\
\st & \\
 & g_i(x)\leq 0 && i=1,\dots,I \label{eq:relaxedFormulationCons1} \\
	& \sum_{j=1}^{\tau} \mu_{j} \phi(z^{j}) \le \zeta && \label{eq:relaxedFormulationCons2} \\
	& \sum_{j=1}^{\tau} \mu_{j}=1 && \label{eq:relaxedFormulationCons3} \\
	& \sum_{j=1}^{\tau} \mu_{j} z^{j}_k=x_k && k=1,\dots,n \label{eq:relaxedFormulationCons4} \\
	& \phi(z^{j}) - \alpha - \sum_{k=1}^{n} \beta_k z_j^k + \gamma_{j} = 0 && j=1,\ldots,\tau \label{eq:relaxedFormulationCons5}\\
	& \mu_{j} \gamma_{j} = 0 && j=1,\ldots \tau \label{eq:relaxedFormulationCons6} \\
	& \mu_{j} \geq 0 && j=1,\ldots \tau \label{eq:relaxedFormulationCons7} \\
	& \gamma_{j} \geq 0 && j=1,\ldots \tau \label{eq:relaxedFormulationCons8}\\
	& x^l_k \le x_k \le x^u_k && k = 1,\ldots,n \label{eq:relaxedFormulationLastCons}
\end{align}
Note that the above program contains product terms in the complementary slackness conditions (\eqref{eq:relaxedFormulationCons6}), which can be linearized using binary variables ($u$) and BigM ($M_1$ and $M_2$) as follows:
\begin{align}
 & \gamma_{j} \leq M_1 u_{j}, \quad j=1,\ldots \tau \label{eq:lin1}\\
& \mu_{j} \leq M_2 (1-u_{j}), \quad j=1,\ldots,\tau \label{eq:lin2}\\
& u_{j} \in \{0,1\}, \quad j=1,\ldots,\tau \label{eq:lin3}
\end{align}
Theorem~\ref{th:bigM} provides tight values for $M_1$ and $M_2$.
\begin{theorem}\label{th:lipschitz}
If $\phi(x)$ is Lipschitz continuous with Lipschitz constant K, then $\hat{\phi}(x|S_{c}): x \in \text{conv } S_{c}$ is also Lipschitz continuous with the maximum possible value of Lipschitz constant as K. 
\end{theorem}
\begin{proof}
From the Lipschitz condition $|\phi(x_1)-\phi(x_2)|\le K||x_1-x_2||$ and the concavity of $\phi(x): x \in \text{conv } S_c$, we can say that $||\omega|| \le K \; \forall \; \omega \in \partial \phi(x)$, where $\partial \phi(x)$ represents the subgradient for $\phi(x)$. The function $\hat{\phi}(x|S_c): x \in \text{conv } S_c$ is a concave polyhedral function, i.e. it consists of piecewise hyperplanes. Consider bounded polyhedra $X_j, j = 1,\ldots,s$ on which the hyperplanes are defined, such that $\nabla \hat{\phi}(x)$ is constant in the interior of $X_j$. Note that $\hat{\phi}(x|S_c) = \phi(x)$ on the vertices, otherwise $\hat{\phi}(x|S_c) \le \phi(x)$. From the property of concavity, it is clear that $\nabla \hat\phi(x) \in \partial \phi(x): x \in X_j$. This implies that $||\nabla \hat\phi(x)|| \le K \; \forall \; x \in X_j$, which can be generalized for all hyperplanes.
%For two points $x_1$ and $x_2$ in $X_j$, $|\phi(x_1)-\phi(x_2)|\le K||x_1-x_2||$, if we show that $|\hat\phi(x_1)-\hat\phi(x_2)|\le K||x_1-x_2||$, the theorem can be proved.
%we can define a function $\eta(x) = \phi(x)-\hat{\phi}(x|S_i)$, such that, $\eta(x) \ge 0, x \in \text{conv } S_i$. 
%To prove the theorem, we will show that the slope of each of the hyperplanes is 
%From Lipschitz condition $|\phi(x_1)-\phi(x_2)|\le K||x_1-x_2||$. Since $\hat{\phi}(x|S_i)$ is a lower (inner) approximation for $\phi(x)$, we can write $\hat{\phi}(x|S_i) \le \phi(x) \; \forall \; x \in \text{conv } S_i$. Also at the approximation points, $\hat{\phi}(z|S_i) = \phi(z)$, $z \in S_i$. Therefore, $|\hat{\phi}(z_1)-\hat{\phi}(z_2)|\le K||z_1-z_2|| \; \forall \; z_1, z_2 \in S_i$. 
%Any non-approximation point $x \in \text{conv} S_i$ can be represented as $\sum_{j=1}^{\tau} \mu_j z_k^j = x_k \; \forall \; k$ and the function value $\hat{\phi(x)}$ can be represented as $\sum_{j=1}^{\tau} \mu_j \phi(z^j)$. See the definition of $\hat{\phi}(x| S_i)$ in \eqref{eq:innerApprox}. Therefore, for any two points $x_1$ and $x_2$, $|\hat\phi(x_1)-\hat\phi(x_2)|$ and $||x_1-x_2||$ can be written as $|\sum_{j=1}^{\tau} \mu_j \phi(z^j)-\sum_{j=1}^{\tau} m_j \phi(z^j)|$ and $||\sum_{j=1}^{\tau} \mu_j z^j - \sum_{j=1}^{\tau} m_j z^j||$, respectively, where $0 \le \mu_j \le 1$ and $0 \le m_j \le 1$ represent the weights. 
%Consider the function $\eta(x) = \phi(x)-\hat{\phi}(x|S_i)$, then $\eta(x) \ge 0$ for $x \in \text{conv } S_i$. Since $\hat{\phi}(x)$ is ...
\end{proof}

\begin{theorem}\label{th:bigM}
$M_1=\psi^{max} = K ||\Delta z||^{max} + \phi^{\text{max}}-\phi^{\text{min}}$ and $M_2=1$ are valid BigM values.
\end{theorem}
\begin{proof}
From constraints \eqref{eq:relaxedFormulationCons3} and \eqref{eq:relaxedFormulationCons7}, we observe that the maximum value that $\mu_j$ can take is $1$. Hence, $M_2=1$ is acceptable. Next, let us look at the proof for $M_1$.

The value of $M_1$ can be set as a value larger than or equal to the maximum possible value that $\gamma_j$ may take. The dual of the lower lever problem in \eqref{eq:bilevelproblem} can be written as
\begin{align}
	\min_{\alpha, \beta, \gamma} & \sum_{k=1}^{n} \beta_k x_k + \alpha \label{eq:objdual}\\
	\st & \sum_{k=1}^{n} \beta_k z_k^j  + \alpha - \gamma_j  = \phi(z^j), && j=1,\dots,\tau \label{eq:consdual1}\\
	& \gamma_j \geq 0, && j=1,\dots,\tau \label{eq:consdual2}
\end{align}
where $\alpha$ is the dual variable for the constraint $\sum_{j=1}^{\tau}\mu_j=1$, $\beta_k$ is the dual for the constraint set $\sum_{j=1}^{\tau} \mu_j z^{j}_k=x_k$, and $\gamma_j$ is the dual for $\mu_j \ge 0$. Note that $\gamma_j$ being a dual for a non-negativity constraint in the primal is essentially a slack variable in the dual formulation. At the optimum, $\beta$ represents the normal to the inner-approximation plane on which the point $(x,\hat\phi(x))$ lies.
At the optimum point of the above dual, at least one slack variable will be equal to zero. Without loss of generality, let us assume that $\gamma_1=0$.
\begin{align}
 \hspace{-30mm} \implies \sum_{k=1}^{n} \beta_k z_k^1  + \alpha = \phi(z^1) \\
 \hspace{-30mm} \implies  \alpha = \phi(z^1)- \sum_{k=1}^{n} \beta_k z_k^1 \label{eq:dualeq1}
\end{align}
By substituting $\alpha$ from \eqref{eq:dualeq1}, we arrive at the following equation for $j=2$
\begin{align}
& \sum_{k=1}^{n}( \beta_k z_k^2 - \beta_k z_k^1) - \gamma_2 = \phi(z^2)-\phi(z^1) \label{eq:dualeq2}\\
%\end{align}
%The left hand side of the above equation can be written in the form of dot product as follows:
%\begin{align}
%& \beta \cdot ( z^2 - z^1) - \gamma_2 = \phi(z^2)-\phi(z^1) \label{eq:dualeq3}\\
\implies & \gamma_2 = \sum_{k=1}^{n}( \beta_k z_k^2 - \beta_k z_k^1) + \phi(z^1)-\phi(z^2) \label{eq:dualeq4}
\end{align}
Given that $\beta$ is a normal to the inner-approximation plane, from Theorem \ref{th:lipschitz}, $\sum_{k=1}^{n}(\beta_k z_k^2 - \beta_k z_k^1) \le K ||z^2 - z^1||$, which implies that
$\gamma_2^{max} = K ||\Delta z||^{max} + \phi^{\text{max}}-\phi^{\text{min}}$.
$K$ is the Lipschitz constant for $\phi$. $\phi^{\text{max}}$ and $\phi^{\text{min}}$ are the maximum and minimum values of $\phi$ in its domain, and $||\Delta z||^{max}$ is the maximum distance between any two values of $z$. Since $z = \{z_k: x^l_k \le z_k \le x^u_k, k = 1,\ldots,n\}$, $||\Delta z||^{max} = \sqrt{\sum_{k=1}^{n} (x^u_k-x^l_k)^2}$.
\end{proof}
%In \eqref{eq:lin1}, $M_1$ can be replaced with a constant $1$, as from constraints \eqref{eq:relaxedFormulationCons3} and \eqref{eq:relaxedFormulationCons7} we observe that the maximum value that $\mu_j$ can take is $1$. In \eqref{eq:lin2}, $M_2$ can be replaced with $\max_x \{\phi(x): x^l_k \le x_k \le x^u_k, \; k = 1,\ldots,n\}$, where $[x^l_k, x^u_k] \;\; \forall \;\; k$ represents the variable bounds for \eqref{eq:objOrig}-\eqref{eq:consConvexOrig}. 
%Variables like $\gamma$ and $\zeta$ can be easily removed from the formulation after making appropriate modifications to get a more concise model.

%\begin{align}
%	\textbf{[EQ-NOPKKT-1]} \quad \min_{x} & \; f(x) + \zeta \label{eq:origobj2}\\
%\st & \; \; \; g_i(x)\leq 0, \quad i=1,\dots,I\\
%	& \zeta \geq \sum_{j=1}^{\tau} \mu_{j} \phi_{1}(x^{j}) \label{eq:origConst1} \\
%	& \sum_{j=1}^{\tau} \mu_{j}=1 \label{eq:origConst2}\\
%	& \sum_{j=1}^{\tau} \mu_{j} x^{j}_k=x_k,\; k=1,\dots,\tau \label{eq:origConst3}\\
%	& \phi_1(x^{j}) + \gamma_{1} + S_{0j} \gamma_{2} - \lambda_{c} = 0 \; \;\forall \;, \;j=1,\ldots,\tau; \; i=1,\ldots \tau; \label{eq:origConst4} \\
%	& \mu_{j} \leq Z_{j}\; \forall \; \; j=1,\ldots \tau \label{eq:origConst5} \\
%	& \lambda_{i} \leq M (1-Z_{j}) \; \forall \; j=1,\ldots,\tau;\;i=1,\ldots \tau; \label{eq:origConst6}\\
%	& Z_{j} \in \{0,1\} \; \forall \; j=1,\ldots,\tau \; \label{eq:origConst7}\\
%	& \lambda_{i} \geq 0 \; \; \forall \; i=1,\ldots \tau \label{eq:origConst8}\\
%	& \gamma_{1},\gamma_{2} \; \; unrestricted \label{eq:origConst9}
%\end{align}

Appropriate values for $M_1$ and $M_2$ are critical for the performance of the method, and therefore the above theorem plays an important role in making the proposed approach competitive. It has been recently shown \citep{kleinert2020there} that identifying a BigM in similar contexts is an NP hard problem, but the structure of the bilevel program in our case allows the identification of tight values for these otherwise large numbers. After linearization of the complimentary slackness conditions, \eqref{eq:relaxedFormulationObj}-\eqref{eq:relaxedFormulationCons5}, \eqref{eq:relaxedFormulationCons7}-\eqref{eq:lin3} is a convex mixed integer program (MIP), which on solving to optimality generates a lower bound for the original program \eqref{eq:objOrig}-\eqref{eq:consConvexOrig}. Solving \eqref{eq:relaxedFormulationObj}-\eqref{eq:relaxedFormulationCons5}, \eqref{eq:relaxedFormulationCons7}-\eqref{eq:lin3} leads to $z^{\tau+1}$ as an optimal point, which is also a feasible point for the original problem \eqref{eq:objOrig}-\eqref{eq:consConvexOrig}. Therefore,substituting $x$ with $z^{\tau+1}$ in \eqref{eq:objOrig} provides an upper bound for the original problem. The optimal point $z^{\tau+1}$ to the convex MIP is used to create a new set $S_{c+1} = S_c \cup z^{\tau+1}$ corresponding to which a new convex MIP is formulated. The new convex MIP formulated with an additional point is expected to provide improved lower and upper bounds in the next iteration of the algorithm. This algorithm is referred to as the Inner-Approximation (IA) algorithm in the rest of the paper. A pseudo-code of IA algorithm is provided in Algorithm 1.
\begin{algorithm} 
	\caption{IA Algorithm for solving concave problem}\label{alg:ConcaveProblem}
	\begin{algorithmic}[1]
		\State \textit{Begin}
		\State $UB_\mathcal{A} \gets +\infty, LB_\mathcal{A} \gets -\infty, c \gets \textit{1}$
		\State $ \text{Choose an initial set of $\tau$ points } S_c = \{ z^{1}, z^{2}, \ldots, z^{\tau} \} $
		\While {(($UB_\mathcal{A}-LB_\mathcal{A})/LB_\mathcal{A} > \epsilon \; ) $ begin} 
		\State Solve Mod-$S_c$ (\eqref{eq:relaxedFormulationObj}-\eqref{eq:relaxedFormulationLastCons}) with a convex MIP solver after linearizing \eqref{eq:relaxedFormulationCons6}
		\State Let the optimal solution for Mod-$S_c$ be $z^{\tau+c}$
		\State $LB_\mathcal{A} \gets f(z^{\tau +c}) + \hat\phi(z^{\tau +c})$ 
		\State $ UB_\mathcal{A} \gets f(z^{\tau +c}) + \phi(z^{\tau +c}) $
		\State $ \mathcal{S}_{c+1} \gets \mathcal{S}_c \; \cup \;z^{\tau+c} $
		\State $c \gets c +1$; 
		\EndWhile
		\State \textit{End}
	\end{algorithmic}
\end{algorithm}
The algorithm starts with an initial set of points $S_1 = \{z^1, \ldots, z^\tau\}$, such that $\text{dom } \phi(x) \subseteq \text{conv } S_1$.

\subsection{The Initial Set}
In this section, we discuss the choice of the initial set $S_1 = \{z^1, \ldots, z^\tau\}$, such that $\text{dom } \phi(x) \subseteq \text{conv } S_1$. The bound constraints in \eqref{eq:objOrig}-\eqref{eq:consConvexOrig} are important so that the initial set $S_1$ may be chosen easily. One of the ways to initialize $S_1$ would be to choose the corner points of the box constraints $x^l_k \le x_k \le x^u_k, \; k = 1,\ldots,n$. Additional points may be sampled randomly between the lower and upper bounds at the start of the algorithm for a better initial approximation of $\phi(x)$, but are not necessary. However, note that for a problem with $n$ variables, choosing the corner points of the box constraints, amounts to starting the algorithm with the cardinality of $S_1$ as $2^n$. For large dimensional problems, the size of the set may be very large, and therefore the approach would be intractable. For large dimensional problem we propose an alternative technique to choose $S_1$, such that $\text{dom } \phi(x) \subseteq \text{conv } S_1$, but the number of points in $S_1$ is only $n+1$.

Without loss of generality, assume that the lower bound is 0 and the upper bound is 1, as one can always normalize the variables by replacing variables $x_k$ with $y_k (x_k^u-x_k^l) + x_k^l$ such that $0 \le y_k \le 1$. In such a case, Figure~\ref{fig:smallPoly} shows the feasible region $g_i(y) \le 0 \; \forall \; i$, enclosed in the polyhedron $0 \le y_k \le 1 \; \forall \; k$. Another polyhedron that encloses $g_i(y) \le 0 \; \forall \; i$ completely is shown in Figure~\ref{fig:largePoly}. While the polyhedron in Figure~\ref{fig:smallPoly} is smaller in terms of the area (or volume), the polyhedron in Figure~\ref{fig:largePoly} is comparatively larger. However, the number of points required to form the polyhedron in Figure~\ref{fig:smallPoly} for an $n$ dimensional problem would be $2^n$, whereas the polyhedron in Figure~\ref{fig:largePoly} will require only $n+1$ points for an $n$ dimensional problem. For the second case, in an $n$ dimensional problem the points can be chosen as follows, $(0,0, \ldots, 0), (n, 0, \ldots, 0), (0, n, \ldots, 0), \ldots, (0, 0, \dots, n)$. These points from the $y$ space can be transformed to the corresponding $x$ space by the following substitution $x_k = y_k (x_k^u-x_k^l) + x_k^l$. One may of course choose any other polyhedron that completely encloses $g_i(x) \le 0 \; \forall \; i$.

	\begin{figure}
		\centering
		\begin{minipage}{.48\textwidth}
			\centering
			\includegraphics[width=.9\linewidth]{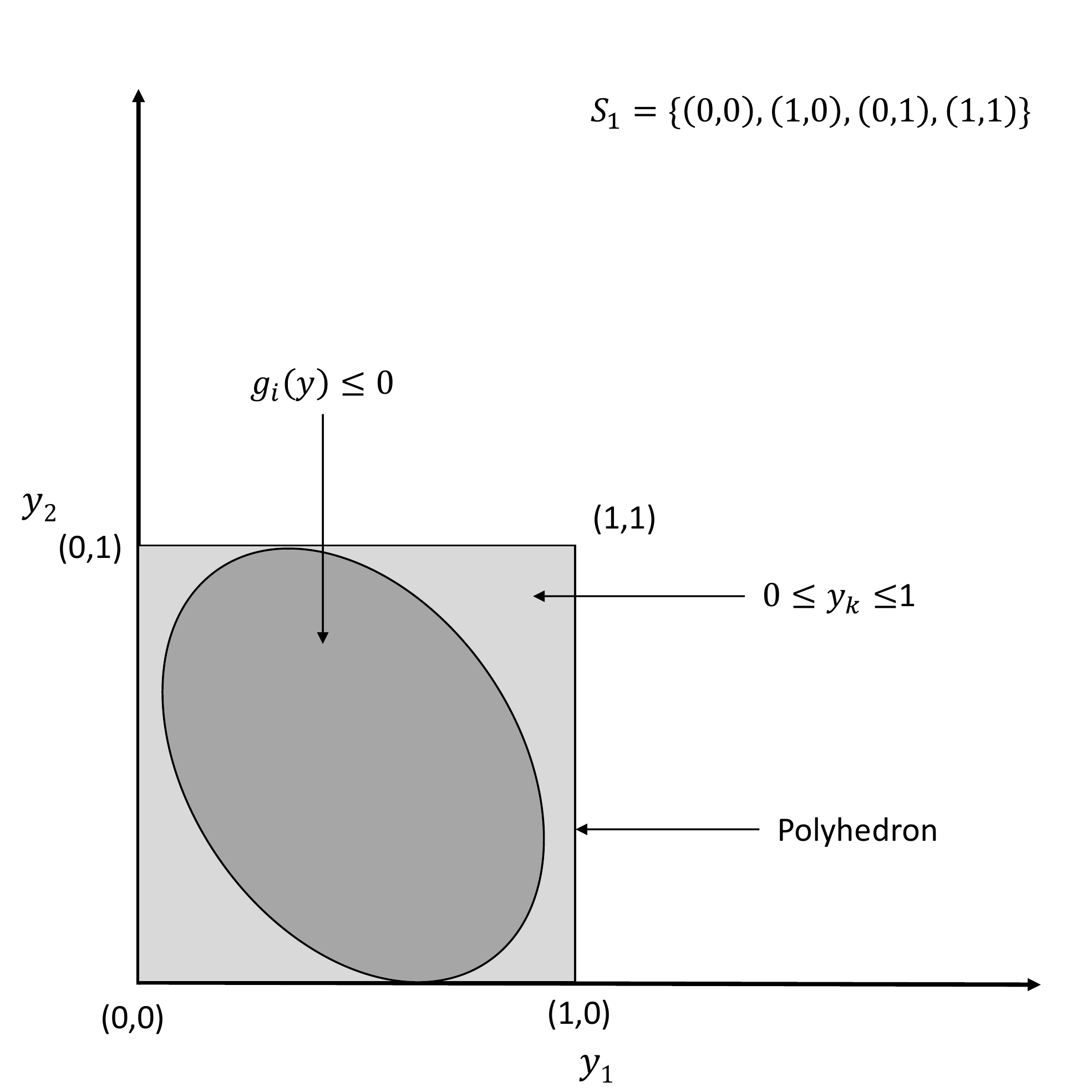}
			\caption{Smaller polyhedron with larger number of points.}
			\label{fig:smallPoly}
		\end{minipage}\hfill
		\begin{minipage}{.48\textwidth}
			\centering
			\includegraphics[width=.9\linewidth]{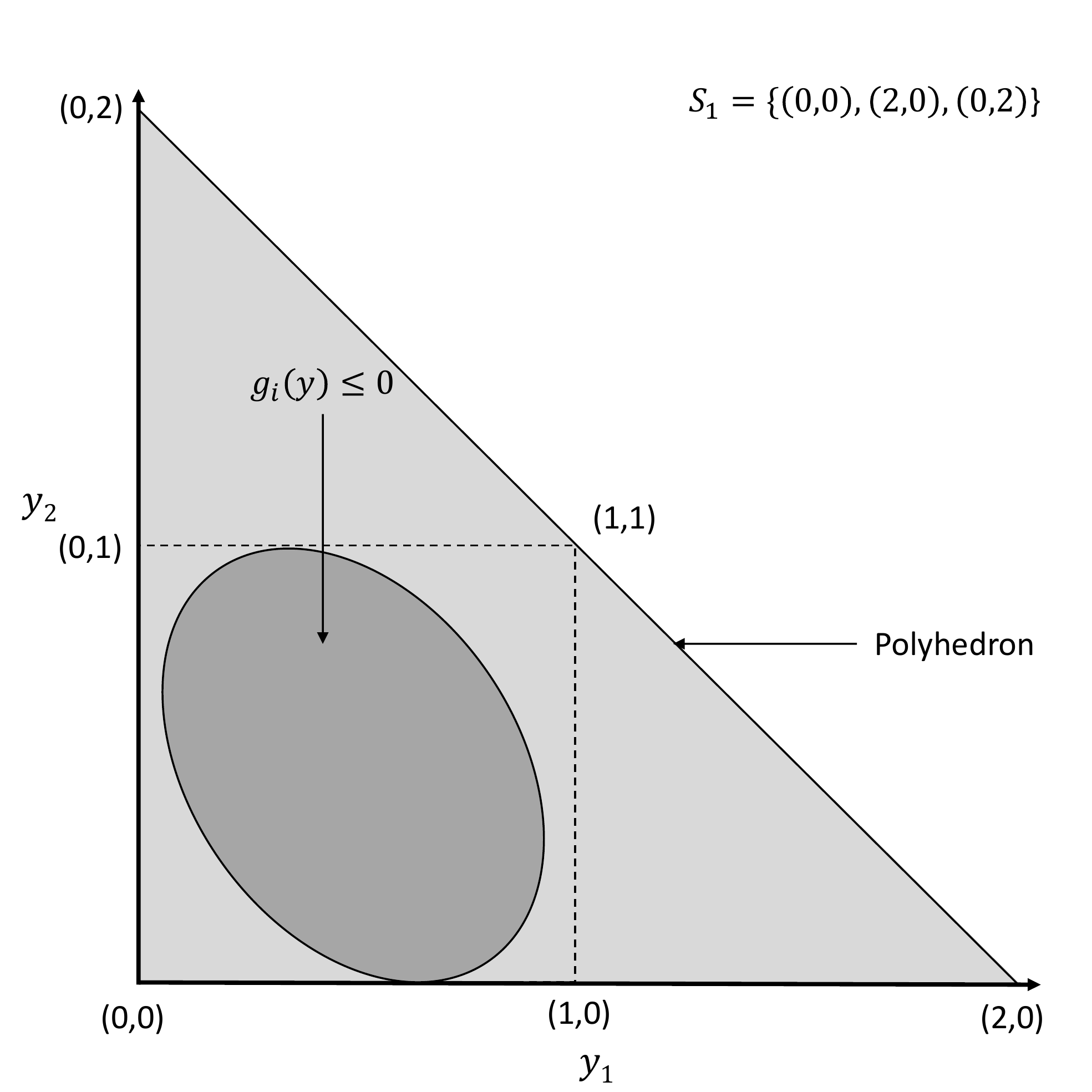}
			\caption{Larger polyhedron with smaller number of points.}
			\label{fig:largePoly}
		\end{minipage}
	\end{figure}

%$y_k = \frac{x_k-x_k^l}{x_k^u-x_k^l}$.

%which should be chosen as the corner points of the box constraints $x^l_k \le x_k \le x^u_k, \; k = 1,\ldots,n$. Additional points may be sampled randomly between the lower and upper bounds at the start of the algorithm for a better initial approximation of $\phi(x)$, but are not necessary.

\subsection{Convergence Results}\label{sec:convergenceProof}
Next, we discuss the convergence results for the proposed algorithm. First we prove that if the algorithm provides the same solution to the lower bound problem (Mod-$S_c$) in two consecutive iterations, then the solution is an optimal solution to the original concave minimization problem \eqref{eq:objOrig}-\eqref{eq:consConvexOrig}.
\begin{theorem}\label{th:equal}
If two consecutive iterations $c$ and $c+1$ lead to the same solution to the lower bound problem (Mod-$S_c$), then the solution is optimal for \eqref{eq:objOrig}-\eqref{eq:consConvexOrig}.
\end{theorem}
\begin{proof}
Say that $z^{\tau+c}$ is the solution at iteration $c$. Note that $S_{c+1} = S_c \cup z^{\tau+c}$, which implies that at iteration $c+1$, $\hat{\phi}(z^{\tau+c}|S_{c+1}) = \phi(z^{\tau+c})$, i.e. when a new point $z^{\tau+c}$ is added in the approximation set, the value of $\hat{\phi}$ and $\phi$ is the same at that point. From Theorem~\ref{th:lowerBound}, at iteration $c+1$, $f(z^{\tau+c+1}) + \hat{\phi}(z^{\tau+c+1}|S_{c+1}) \le f(z^{\tau+c+1}) + \phi(z^{\tau+c+1})$. Since $z^{\tau+c+1}=z^{\tau+c}$ and $\hat{\phi}(z^{\tau+c}|S_{c+1}) = \phi(z^{\tau+c})$, $f(z^{\tau+c+1}) + \hat{\phi}(z^{\tau+c+1}|S_{c+1}) \le f(z^{\tau+c+1}) + \phi(z^{\tau+c+1})$ holds with an equality implying that $z^{\tau+c}$ is the optimal solution.
\end{proof}

\begin{theorem}
When the algorithm proceeds from iteration $c$ to $c+1$, then the lower bound for \eqref{eq:objOrig}-\eqref{eq:consConvexOrig} improves if $z^{\tau+c}$ at iteration $c$ is not the optimum for \eqref{eq:objOrig}-\eqref{eq:consConvexOrig}.
\end{theorem}
\begin{proof}
It is given that $z^{\tau+c}$ is the solution for \eqref{eq:objRel}-\eqref{eq:consRel} at iteration $c$, which is not optimal for the original problem (\eqref{eq:objOrig}-\eqref{eq:consConvexOrig}). Say that $z^{\tau+c+1}$ is the solution for \eqref{eq:objRel}-\eqref{eq:consRel} at iteration $c+1$, so from Theorem~\ref{th:equal} we can say that $z^{\tau+c} \ne z^{\tau+c+1}$. 

Note that for any given $x$, $\hat{\phi}(x|S_{c}) \le \hat{\phi}(x|S_{c+1})$, as the linear program corresponding to $\hat{\phi}(x|S_{c+1})$ is a relaxation of $\hat{\phi}(x|S_{c})$. This is shown in the next statement. If $\mu_{\tau+1} = 0$ is added in the linear program corresponding to $\hat{\phi}(x|S_{c+1})$, it becomes equivalent to the linear program corresponding to $\hat{\phi}(x|S_{c})$, which shows that $\hat{\phi}(x|S_{c+1})$ is a relaxation of $\hat{\phi}(x|S_{c})$.

Since $\hat{\phi}(x|S_{c}) \le \hat{\phi}(x|S_{c+1})$ for all $x$, we can say that $f(x) + \hat{\phi}(x|S_{c}) \le f(x) + \hat{\phi}(x|S_{c+1})$ for all $x$. This implies that for \eqref{eq:objRel}-\eqref{eq:consRel} comparing the objective function, we get $f(z^{\tau+c}) + \hat{\phi}(z^{\tau+c}) \le f(z^{\tau+c+1}) + \hat{\phi}(z^{\tau+c+1})$. Strict concavity of $\phi$ and $z^{\tau+c} \ne z^{\tau+c+1}$ ensure that the equality will not hold, implying $f(z^{\tau+c}) + \hat{\phi}(z^{\tau+c}) < f(z^{\tau+c+1}) + \hat{\phi}(z^{\tau+c+1})$. Therefore, the lower bound strictly improves in the next iteration.
\end{proof}

\begin{theorem}\label{th:termination}
If $z^{\tau+c}$ is the solution for \eqref{eq:objRel}-\eqref{eq:consRel} at iteration $c$, $z^{\tau+c+1}$ is the solution for \eqref{eq:objRel}-\eqref{eq:consRel} at iteration $c+1$, and $||z^{\tau+c+1}-z^{\tau+c}|| \le \delta$, then the optimal function value $v^{\ast}$ for \eqref{eq:objOrig}-\eqref{eq:consConvexOrig} has the following property: 
$0 \le f(z^{\tau+c+1})+\hat{\phi}(z^{\tau+c+1}|S_{c+1}) - v^{\ast} \le (K_1+K_2)\delta$, where $K_1$ and $K_2$ are the Lipchitz constants for $f(x)$ and $\phi(x)$, respectively.
%$f(z^{\tau+i+1})+\hat{\phi}(z^{\tau+i+1}|S_{i+1}) \le v^{\ast} \le f(z^{\tau+i+1})+\hat{\phi}(z^{\tau+i+1}|S_{i+1}) + (K_1+K_2)\delta$, where $K_1$ and $K_2$ are the Lipchitz constants for $f(x)$ and $\phi(x)$, respectively.
%$f(z^{\tau+i})+\hat{\phi}(z^{\tau+i}|S_{i+1}) - (K_1+K_2)\delta \le v^{\ast} \le f(z^{\tau+i})+\hat{\phi}(z^{\tau+i}|S_{i+1})$
\end{theorem}
\begin{proof}
Note that at iteration $c+1$, $\hat{\phi}(z^{\tau+c}|S_{c+1}) = \phi(z^{\tau+c})$ with $z^{\tau+c}$ being a feasible solution for \eqref{eq:objOrig}-\eqref{eq:consConvexOrig}. Therefore, $f(z^{\tau+c})+\hat{\phi}(z^{\tau+c}|S_{c+1})$ is the upper bound for $v^{\ast}$. Also $f(z^{\tau+c+1})+\hat{\phi}(z^{\tau+c+1}|S_{c+1})$ is the lower bound for $v^{\ast}$ from Theorem~\ref{th:lowerBound}.
$$
f(z^{\tau+c+1})+\hat{\phi}(z^{\tau+c+1}|S_{c+1}) \le v^{\ast} \le f(z^{\tau+c})+\hat{\phi}(z^{\tau+c}|S_{c+1})
$$
If $K_2$ is the Lipschitz constant for $\phi(x)$, then it is also the Lipschitz constant for $\hat{\phi}(x)$ from Theorem~\ref{th:lipschitz}. From the Lipschitz property, 
$$
f(z^{\tau+c+1})+\hat{\phi}(z^{\tau+c+1}|S_{c+1}) \le v^{\ast} \le f(z^{\tau+c})+\hat{\phi}(z^{\tau+c}|S_{c+1}) \le f(z^{\tau+c+1})+\hat{\phi}(z^{\tau+c+1}|S_{c+1}) + (K_1+K_2)\delta
$$
which implies
$0 \le f(z^{\tau+c+1})+\hat{\phi}(z^{\tau+c+1}|S_{c+1}) - v^{\ast} \le (K_1+K_2)\delta$.
%$f(z^{\tau+i+1})+\hat{\phi}(z^{\tau+i+1}|S_{i+1}) \le v^{\ast} \le f(z^{\tau+i+1})+\hat{\phi}(z^{\tau+i+1}|S_{i+1}) + (K_1+K_2)\delta$
\end{proof}

%\begin{theorem}
%Let $\{z^{\tau+k}\}$ be a sequence generated by the algorithm, then for any minimizer $z^{\ast}$, such that $v^{\ast}=f(z^{\ast})+\phi(z^{\ast})$, and $k\ge 1$, we have the following, where C is a constant.
%$$
%[f(z^{\tau+k})+\phi(z^{\tau+k}] - [f(z^{\ast})+\phi(z^{\ast})] \le \frac{C||z^{\ast}-z^{\tau+1}||}{k}
%$$
%\end{theorem}
%\begin{proof}
%Assume that an $\epsilon$-termination is desired such that $[f(z^{\tau+k})+\phi(z^{\tau+k}] - [f(z^{\ast})+\phi(z^{\ast})] \le \epsilon$, and $\epsilon=(K_1+K_2)\delta$. From Theorem~\ref{th:termination}, the termination will happen at $k$, when $z^{\tau+k}-z^{\tau+k-1} \le \delta$.
%%% Unable to prove sublinear convergence, i.e. $[f(z^{\tau+k})+\phi(z^{\tau+k}] - [f(z^{\ast})+\phi(z^{\ast})] \sim O(1/k)$
%\end{proof}
To illustrate the working of the algorithm, an example has been provided in the Appendix (see Section~\ref{sec:NumEx}). Next, we apply the algorithm on two common classes of concave minimization problems.

\section{Concave Knapsack Problem}\label{sec:knapsackprob}
The integer/binary knapsack problem requires determining the items to be chosen from a given collection of items with certain weights and values so as to maximize the total value without exceeding a given total weight limit. Over the last sixty years, integer/binary knapsack problems have received considerable attention mostly due to their wide variety of applications in financial decision problems, knapsack cryptosystems, combinatorial auctions, etc. \citep{kellerer2004some}. The integer/binary Knapsack problem is known to be NP-complete, for which a variety of algorithms have been reported in the literature, including Lagrangian relaxation \citep{fayard1982algorithm, fisher2004Lagrangian}, branch-and-bound (B$\&$B) \citep{kolesar1967branch}, dynamic programming \citep{martello1999dynamic}, and hybrid methods combining B$\&$B and dynamic programming \citep{marsten1978hybrid}, cut-and-branch algorithm \citep{fomeni2020cut}, etc. The literature has also seen a proliferation of papers on non-linear Knapsack problems (NKP), which arise from economies and dis-economies of scale in modelling various problems such as capacity planning \citep{bitran1989tradeoff}, production planning \citep{ziegler1982solving,ventura1988note,maloney1993constrained}, stratified sampling problems \citep{bretthauer1999nonlinear}, financial models \citep{mathur1983branch}, etc. NKP also arises as a subproblem in solving service system design problems and facility location problems with stochastic demand \citep{elhedhli2005exact}. NKP may be a convex or a non-convex problem in nature. Each of these types can be further classified as continuous or integer knapsack problems, separable or non-separable knapsack problems. In this paper, we aim to solve the concave separable integer knapsack problem (CSINK), where concavity in the objective function arises due to the concave cost structure. There are a plethora of applications that involve concave costs, such as capacity planning and fixed charge problems with integer variables \citep{bretthauer1995nonlinear, horst1998integer, horst2013global}, and other problems with economies of scale \citep{pardalos1987constrained}. Specifically, applications of CSINK include communication satellite selection \citep{witzgall1975mathematical}, pluviometer selection in hydrological studies \citep{gallo1980quadratic,caprara1999exact}, compiler design \citep{johnson1993min, pisinger2007quadratic}, weighted maximum b-clique problems \citep{park1996extended, dijkhuizen1993cutting, pisinger2007quadratic, caprara1999exact}. Due to its wide applications, CSINK has attracted a lot of researchers to solve it efficiently.\par
\cite{gallo1980quadratic} reported one of the first approaches for the quadratic knapsack problem by utilizing the concept of the upper plane, which is generated by the outer linearization of the concave function. Researchers have also come up with different B$\&$B-based algorithms to solve the concave minimization version of the problem with integer variables \citep{marsten1978hybrid,victor1986branch, benson1990algorithm, bretthauer1994algorithm, caprara1999exact}. \cite{chaillou1989best} proposed a Lagrangian relaxation-based bound of the quadratic knapsack problem. \cite{more1990solution} proposed an algorithm that characterizes local minimizers when the objective function is strictly concave and used this characterization in determining the global minimizer of a continuous concave knapsack problem with linear constraints. \cite{michelon1996Lagrangean} reported a Lagrangian based decomposition technique for solving the concave quadratic knapsack problem. Later, \cite{sun2005exact} developed an iterative procedure of linearly underestimating the concave function and executing domain cut and partition by utilizing the special structure of the problem for solving CSINK. They used their proposed algorithm to solve only a single-dimensional CSINK, i.e., one with only one constraint.   
%with single linear constraint. However, the paper did not discuss the multiple constrained case.  
Most recently, \cite{wang2019new} reported an exact algorithm that combines the \emph{contour cut} \citep{li2006convergent} with a special cut to gradually reduce the duality gap through an iterative process to solve a multidimensional CSINK, i.e., one with multiple knapsack constraints. Wang showed that his proposed algorithm outperformed the hybrid method proposed by \cite{marsten1978hybrid}, which to the best of our knowledge, is the only other exact method to solve multidimensional CSINK.\par
The model for CSINK is described below:

\begin{align}
	\min_{x} & \; \; \phi(x) = \sum_{j=1}^n \phi_j(x_j) \label{eq:objknap}\\
	\st & \notag\\
	& Ax \leq b \label{eq:consknap1}\\
	& x \in X = \{x \in \mathbb{Z}^n \; | \; l_j \leq x_j \leq u_j \label{eq:consknap2}\}
\end{align}
where $\phi_j(x_j), \; j=1.\ldots n$ are concave non-decreasing functions, $A \in \mathbb{R}^{m \times n}$, $b \in \mathbb{R}^m$ and $l = (l_1,\ldots l_n)^T, u = (u_1,\ldots u_n)^T$ are lower and upper bounds of $x$ respectively. For $m\geq 1$, the above problem is commonly referred to as multidimensional CSINK. In the next section, we discuss the data-sets used for the computational experiments and present the results of the IA algorithm to solve multidimensional CSINK. We benchmark our method against \cite{wang2019new}.

%============================================
\subsection{Computational Experiments}\label{sec:compuExperKnapsack}
%============================================
In this section, we present the data generation technique, followed by a discussion on computational results. All computational experiments are carried out on a PC with Pentium(R) Dual-core CPU i5-6200U @2.3 GHz and 8 GB RAM. As described in Section 2, the IA algorithm is coded in C++, and the MIP in step 5 of Algorithm~\ref{alg:ConcaveProblem} is solved using the default Branch\&Cut solver of CPLEX 12.7.1. The optimality gap ($\epsilon=0.01$) is calculated as $\frac{UB-LB}{LB}~\times~100$, where $UB$ and $LB$ denote the upper bound and lower bound for the original problem, respectively. The algorithm is set to terminate using $\epsilon=0.01$ in step 4 of Algorithm~\ref{alg:ConcaveProblem} or using a CPU time limit of 2 hours, whichever reaches first. We compare the computational performance of our method against \cite{wang2019new}. The experiments by \cite{wang2019new} are done on a PC with Pentium(R) Dual-core CPU E6700 @3.2GHz, which is approximately $2.79$ times slower than our system (https://www.cpubenchmark.net/singleCompare.php). Hence, for a fair comparison, we scale the computational times of the IA algorithm by a factor of $2.79$.\par
\subsubsection{Data-Set}
All our computational experiments are performed on random data-sets, generated using the following scheme as described by \cite{wang2019new}. In all the test data-sets: $A=\{a_{ij}\}_{n \times m} \in [-20,-10]$; $b_i=\sum_{j=1}^n a_{ij}l_{j} + r \left(\sum_{j=1}^{n} a_{ij}u_{j} - \sum_{j=1}^{n} a_{ij}l_{j}\right)$; where $r=0.6$; and $l_j=1, u_j=5;n\in \{30,\ldots, 150\}, m \in\{10,15\} $. Further, we employ two different forms of concavity in the objective function \eqref{eq:objknap}: (i) polynomial form; and (ii) non-polynomial form. The parameters settings for both categories are briefly described as follows. 

\begin{enumerate}
[(i)]\item Polynomial concave function:
\begin{align*}
\phi(x)= \sum_{j=1}^n (c_j x_j^4 + d_j x_j^3+e_jx_j^2+h_jx_j)
\end{align*}
We use the following three kinds of polynomial concave functions, as used by \cite{wang2019new}. For a fixed pair $n$ and $m$, ten random test problems are generated from a uniform distribution using the following scheme.
\begin{itemize}
\item Quadratic: $c_j =0, d_j =0,e_j \in [-15,-1], h_j \in [-5,5], j = 1,\ldots, n$.
 \item Cubic: $c_j =0, d_j \in (-1,0),e_j \in [-15,-1] , h_j \in [-5,5], j = 1,\ldots, n$.
 \item Quartic: $c_j \in (-1,0), d_j \in (-5,0), e_j \in [-15,-1] , h_j \in [-5,5], j = 1,\ldots, n$
\end{itemize} 
\item Non-polynomial concave function:
\begin{align*}
\phi(x)= \sum_{j=1}^n (c_j ln(x_j) +d_jx_j) 
\end{align*}
 %Here we consider solving concave knapsack problem where the concave function is in logarithmic form. 
 Once again, we generate 10 random data instances for a fixed pair $n$ and $m$ using uniform distribution with the following parameters: $c_j \in (0,1), d_j \in [-20,-10],j=1,\ldots ,n.$
\end{enumerate}
%===========================================
\subsubsection{Computational Results}\label{sec:compuResultsKnapsack}
\begin{table}[htbp]
  \centering
  \begin{threeparttable}
  \caption{Experimental results for quadratic concave knapsack problem}
    \begin{tabular}{ccccccccc}
    \toprule
          &       & \multicolumn{7}{c}{CPU Time (seconds)} \\
\cmidrule{3-9}          &       & \multicolumn{3}{c}{IA Algorithm$^{\ast}$} &       & \multicolumn{3}{c}{\cite{wang2019new}} \\
\cmidrule{3-5}\cmidrule{7-9}   n$\times$m  &       & Avg   & Min   & Max   &       & Avg   & Min   & Max \\
    \midrule
    30$\times$10 &       & \textbf{5.62} & 0.32  & 20.48 &       & 17.85 & 0.41  & 75.64 \\
    40$\times$10 &       & \textbf{3.80} & 0.16  & 10.24 &       & 50.98 & 2.05  & 350.94 \\
    50$\times$10 &       & \textbf{3.99} & 0.81  & 12.74 &       & 142.39 & 1.34  & 980.34 \\
    80$\times$10 &       & \textbf{19.32} & 1.61  & 40.32 &       & 793.83 & 14.81 & 7212.39 \\
    150$\times$10 &       & \textbf{12.65} & 0.48  & 46.37 &       & 1116.95 & 0.01  & 3232.39 \\
    20$\times$15 &       & \textbf{5.33} & 0.48  & 17.66 &       & 31.77 & 2.88  & 237.75 \\
    30$\times$15 &       & \textbf{44.67} & 0.81  & 132.10 &       & 140.91 & 1.14  & 587.64 \\
    40$\times$15 &       & \textbf{87.96} & 1.77  & 241.38 &       & 710.48 & 2.45  & 3125.30 \\
    \midrule
    Avg   &       & \textbf{22.92} & 0.81  & 65.16 &       & 375.65 & 3.14  & 1975.30 \\
    \bottomrule
    \end{tabular}%
 \begin{tablenotes}[normal,flushleft]
\item $^{\ast}$Original CPU times are scaled by 2.79 for a fair comparison.
\end{tablenotes}
 \label{table:tabquad}%
 \end{threeparttable}
\end{table}%
%#################################

%#################################
% Table generated by Excel2LaTeX from sheet 'knapsack result compari final'
% Table generated by Excel2LaTeX from sheet 'knapsack result compari_bigM1.3'
\begin{table}[htbp]
  \centering
   \begin{threeparttable}
  \caption{Experimental results for cubic concave knapsack problem}
    \begin{tabular}{ccccccccc}
    \toprule
          &       & \multicolumn{7}{c}{CPU Time (seconds)} \\
\cmidrule{3-9}          &       & \multicolumn{3}{c}{IA Algorithm$^{\ast}$} &       & \multicolumn{3}{c}{Wang's Method} \\
\cmidrule{3-9}    n$\times$m  &       & Avg   & Min   & Max   &       & Avg   & Min   & Max \\
\cmidrule{1-1}\cmidrule{3-9}    30$\times$10 &       & \textbf{10.74} & 0.25  & 35.82 &       & 10.91 & 0.25  & 26.80 \\
    40$\times$10 &       & \textbf{10.98} & 0.18  & 34.50 &       & 30.73 & 1.30  & 98.47 \\
    60$\times$10 &       & \textbf{31.18} & 0.93  & 92.38 &       & 166.62 & 6.72  & 1002.75 \\
    80$\times$10 &       & \textbf{182.98} & 1.63  & 876.89 &       & 275.28 & 5.08  & 1672.88 \\
    90$\times$10 &       & \textbf{209.05} & 0.30  & 1169.80 &       & 631.08 & 0.00  & 5457.95 \\
    20$\times$15 &       & \textbf{29.51} & 1.47  & 117.04 &       & 153.48 & 1.41  & 1059.48 \\
    30$\times$15 &       & \textbf{102.01} & 4.60  & 267.33 &       & 159.91 & 3.58  & 999.77 \\
    50$\times$15 &       & \textbf{858.22} & 65.34 & 3432.73 &       & 2172.94 & 36.08 & 12747.97 \\
    \midrule
    Avg   &       & \textbf{179.33} & 9.34  & 753.31 &       & 450.12 & 6.80  & 2883.26 \\
    \bottomrule
    \end{tabular}%
  \begin{tablenotes}[normal,flushleft]
\item $^{\ast}$Original CPU times are scaled by 2.79 for a fair comparison.
\end{tablenotes}
\label{table:tabcubic}%
 \end{threeparttable}
\end{table}%

%==================================
Tables~\ref{table:tabquad}-\ref{table:tablog} provide a comparison of the computational performance of the IA algorithm against  \cite{wang2019new}. Since \cite{wang2019new} report only the average, minimum, and maximum CPU times over $10$ randomly generated instances for each size of the problem, we also do the same for a meaningful comparison. It is important to highlight that, as discussed earlier in this section, for a fair comparison, the CPU times for the IA algorithm have been scaled by a factor of $2.79$ before reporting in Tables~\ref{table:tabquad}-\ref{table:tablog}. For each problem size, the better of the two average CPU times (one for the IA algorithm and the other for \cite{wang2019new}) is highlighted in boldface. Tables~\ref{table:tabquad}-\ref{table:tabquar} provide the results corresponding to the three different polynomial forms (quadratic, cubic and quartic). As evident from the tables, the IA algorithm consistently outperforms \cite{wang2019new} over all the instances for the case of the quadratic objective function \eqref{eq:objknap}, and over most of the instances for the other forms of the objective function except for the few easy instances. Specifically, for the quadratic objective function, the IA algorithm takes 22.92 seconds on average, over all the instances, which is less than one-sixteenth of 375.65 required by \cite{wang2019new}. For the cubic and the quartic form of the objective function, the average times over all the instances taken by the IA algorithm are 179.33 and 101.73, respectively, while the average times taken by \cite{wang2019new} are 450.12 and 259.32. Table~\ref{table:tablog} provides the results corresponding to the non-polynomial (logarithmic) form of the objective function \eqref{eq:objknap}. Clearly, the IA algorithm consistently outperforms \cite{wang2019new} over all the instances for the case of the logarithmic objective function, taking an average of only 1.48 seconds, which is around 114 times smaller than the average time of 169.18 seconds taken by \cite{wang2019new}.\par
%====================================
% Table generated by Excel2LaTeX from sheet 'knapsack result compari final'
% Table generated by Excel2LaTeX from sheet 'knapsack result compari_bigM1.3'
\begin{table}[ptbh]
  \centering
  \begin{threeparttable}
  \caption{Experimental results for quartic concave knapsack problem}
    \begin{tabular}{ccccccccc}
\toprule          &       & \multicolumn{7}{c}{CPU Time (seconds)} \\
\cmidrule{3-9}          &       & \multicolumn{3}{c}{IA Algorithm$^{\ast}$} &       & \multicolumn{3}{c}{\cite{wang2019new}} \\
\cmidrule{3-5}\cmidrule{7-9}    n$\times$m  &       & Avg   & Min   & Max   &       & Avg   & Min   & Max \\
    \midrule
    30$\times$10 &       & 23.22 & 0.88  & 89.33 &       & \textbf{13.09} & 2.02  & 35.81 \\
    50$\times$10 &       & 57.44 & 0.16  & 223.17 &       & \textbf{44.92} & 4.61  & 153.17 \\
    70$\times$10 &       & \textbf{41.63} & 1.17  & 200.05 &       & 396.58 & 8.03  & 1994.47 \\
    100$\times$10 &       & \textbf{145.29} & 3.46  & 968.65 &       & 611.00 & 8.80  & 4963.13 \\
    20$\times$15 &       & \textbf{77.74} & 7.24  & 272.28 &       & 117.90 & 3.03  & 402.61 \\
    30$\times$15 &       & \textbf{88.55} & 0.88  & 434.04 &       & 188.39 & 3.22  & 1583.52 \\
    40$\times$15 &       & \textbf{278.23} & 0.57  & 1081.13 &       & 443.32 & 6.09  & 1281.44 \\
    \midrule
    Avg   &       & \textbf{101.73} & 2.05  & 466.95 &       & 259.32 & 5.11  & 1487.73 \\
    \bottomrule
    \end{tabular}%
 \begin{tablenotes}[normal,flushleft]
\item $^{\ast}$Original CPU times are scaled by 2.79 for a fair comparison.
\end{tablenotes}
 \label{table:tabquar}%
  \end{threeparttable}
\end{table}%

% Table generated by Excel2LaTeX from sheet 'knapsack result compari_bigM1.3'
\begin{table}[ptbh]
  \centering
  \begin{threeparttable}
  \caption{Experimental results for logarithmic concave knapsack problem($\phi(x)=\sum_{j=1}^n (c_j ln(x_j) +d_jx_j)$)}
    \begin{tabular}{ccccccccc}
    \toprule
          &       & \multicolumn{7}{c}{CPU Time (seconds)} \\
\cmidrule{3-9}          &       & \multicolumn{3}{c}{IA Algorithm$^{\ast}$} &       & \multicolumn{3}{c}{\cite{wang2019new}} \\
\cmidrule{3-5}\cmidrule{7-9}    n$\times$m  &       & Avg   & Min   & Max   &       & Avg   & Min   & Max \\
    \midrule
    30$\times$10 &       & \textbf{0.32} & 0.11  & 0.70  &       & 3.97  & 0.20  & 18.59 \\
    50$\times$10 &       & \textbf{0.43} & 0.10  & 0.88  &       & 9.54  & 1.08  & 24.11 \\
    70$\times$10 &       & \textbf{0.47} & 0.10  & 1.13  &       & 44.11 & 1.86  & 147.33 \\
    95$\times$10 &       & \textbf{0.61} & 0.14  & 3.08  &       & 277.36 & 0.02  & 1484.20 \\
    30$\times$15 &       & \textbf{1.05} & 0.12  & 4.59  &       & 12.70 & 0.94  & 39.55 \\
    50$\times$15 &       & \textbf{2.43} & 0.32  & 9.69  &       & 289.21 & 9.11  & 1156.89 \\
    70$\times$15 &       & \textbf{5.07} & 0.35  & 28.24 &       & 547.38 & 27.03 & 2100.30 \\
    \midrule
    Avg   &       & \textbf{1.48} & 0.18  & 6.90  &       & 169.18 & 5.75  & 710.14 \\
    \bottomrule
    \end{tabular}%
 \begin{tablenotes}[normal,flushleft]
\item $^{\ast}$Original CPU times are scaled by 2.79 for a fair comparison.
\end{tablenotes}
 \label{table:tablog}%
  \end{threeparttable}
\end{table}%

%===================================
To further see the difference in the performances of two methods, we present their performance profiles \citep{dolan2002benchmarking} in Figures \ref{fig:quadavg}-\ref{fig:logavg}. For this, let $t_{p,s}$ represent the CPU time to solve instance $p \in P$ using method $s \in S$. Using this notation, the performance ratio ($r_{p,s}$), which is defined as the ratio of the CPU time taken by a given method to that taken by the best method for that instance, can be mathematically given as follows:\\
\begin{equation}
r_{p,s} = \frac{t_{p,s}}{\min_{s \in S} t_{p,s}}
\end{equation}
If we assume $r_{p,s}$ as a random variable, then the performance profile $(p_{s}(\tau) )$ is the cumulative distribution function of $r_{p,s}$ at $2^{\tau}$, mathematically expressed as $p_{s}(\tau)=P(r_{p,s} \leq 2^{\tau})$. In other words, it gives the probability that the CPU time taken by the method $p$ does not exceed $2^{\tau}$ times that taken by the best of the two methods. Further, for a given method $p$, the intercept of its performance profile on the y-axis shows the proportion of the instances for which it performs the best. The performance profiles for the polynomial functions are displayed in Figures \ref{fig:quadavg}-\ref{fig:quarticavg}, and the same for the non-polynomial function are displayed in Figure \ref{fig:logavg}. In the absence of the CPU times for each individual instance by \cite{wang2019new}, we use the average computational times (after scaling by a factor of 2.79) for creating the performance profiles. 
%revise from here
From Figures \ref{fig:quadavg}-\ref{fig:quarticavg}, it can be concluded that the IA algorithm outperforms \cite{wang2019new} for 100\% of the instances for the quadratic objective function and cubic objective function, while it is better for 71.42\% of the instances for the quartic objective functions. For the non-polynomial (logarithmic) form of the objective function, Figure~\ref{fig:logavg} shows the IA algorithm as outperforming \cite{wang2019new} for 100\% of the data instances. Furthermore, for the instances (for quadratic and non-polynomial objective functions) on which the performance of \cite{wang2019new} is worse than that of the IA algorithm, it is unable to solve them to optimality even after $16=(2^4)$ times the CPU time taken by the IA algorithm. Next, we discuss the formulation of the production-transportation problem and report the results of the IA algorithm benchmarking it against two approaches.
%The difference in the computational performances between the two methods becomes even more prominent when we compare the computational performances only over the instances for which the IA algorithm is faster. 
% From Figures \ref{fig:quadavg} and \ref{fig:logavg}, it can be inferred that \cite{wang2019new} is unable to converge to the global optima for all data instances even after $16=(2^4)$ times the computational time taken by IA algorithm. Next, we apply the IA algorithm to another class of concave minimization problems.
%Similarly, from Figures 6-7, it can be inferred that for the data instances in which the proposed algorithm is faster, \cite{wang2019new}'s algorithm is unable to converge even after 16 times the computational time taken by IA algorithm.
%===================================
	\begin{figure}[ptbh]
		\centering
		\begin{minipage}{.48\textwidth}
			\centering
			\includegraphics[width=1.13\linewidth]{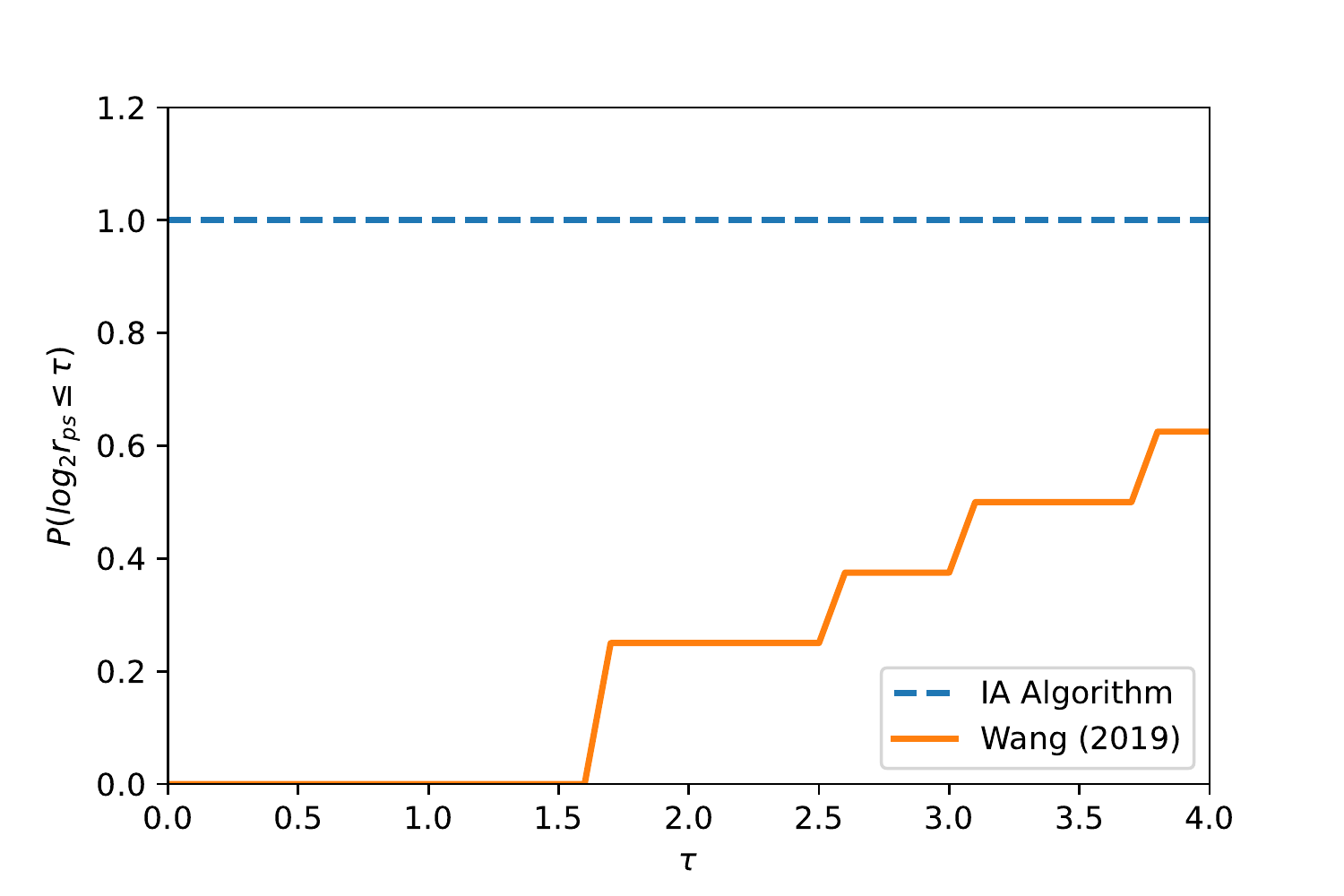}
			\caption{Performance profile of quadratic knapsack problem}
			\label{fig:quadavg}
		\end{minipage}\hfill
		\begin{minipage}{.48\textwidth}
			\centering
			\includegraphics[width=1.13\linewidth]{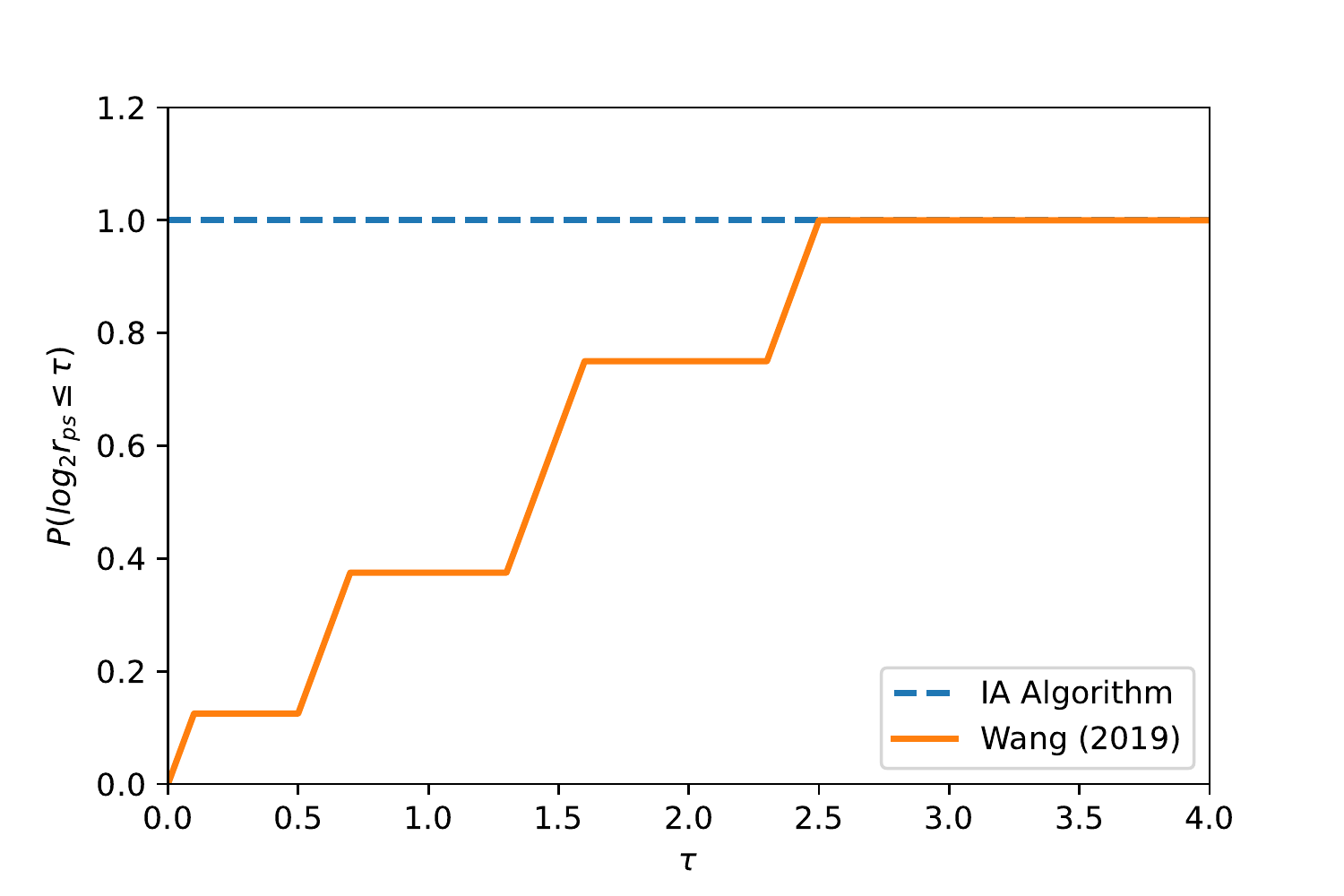}
			\caption{Performance profile of cubic knapsack problem}
			\label{fig:cubicavg}
		\end{minipage}
	\end{figure}
	\begin{figure}
		\centering
		\begin{minipage}{.48\textwidth}
			\centering
			\includegraphics[width=1.13\linewidth]{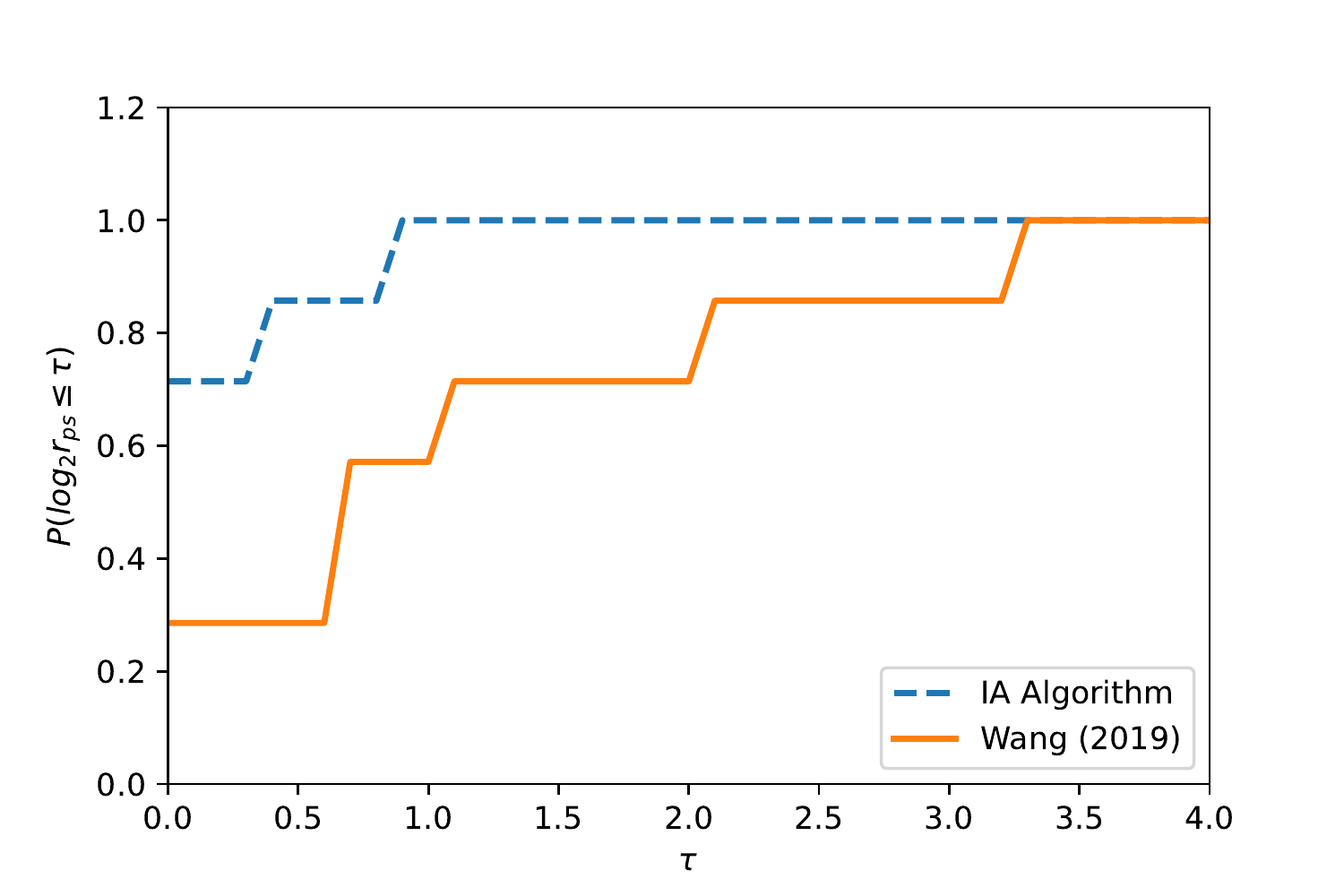}
			\caption{Performance profile of quartic knapsack problem}
			\label{fig:quarticavg}
		\end{minipage}\hfill
		\begin{minipage}{.48\textwidth}
			\centering
			\includegraphics[width=1.13\linewidth]{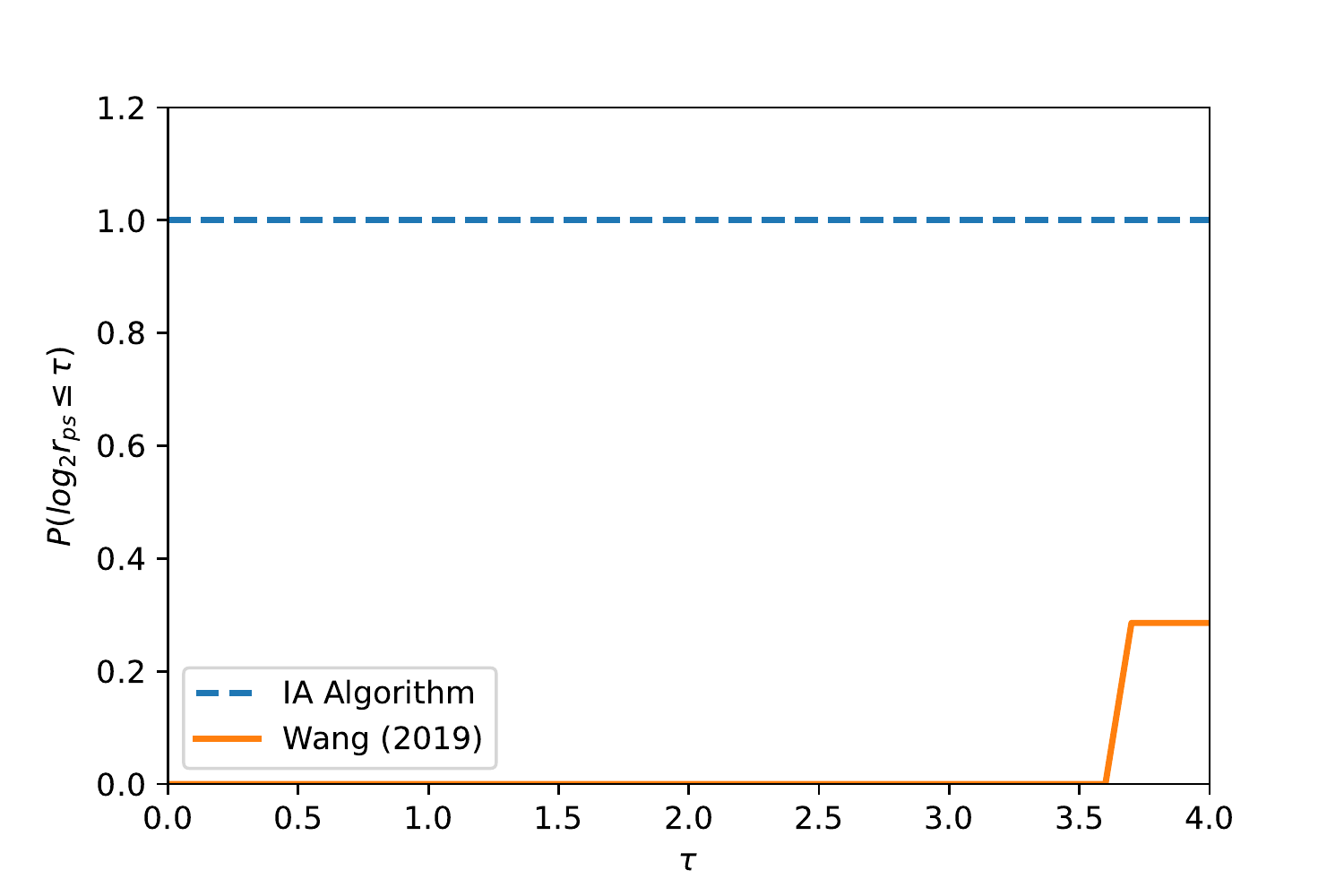}
			\caption{Performance profile of log knapsack problem}
			\label{fig:logavg}
		\end{minipage}
	\end{figure}
%==================================================	
\section{Production-Transportation Problem}\label{sec:prodtransprob}
%==================================================
The transportation problem is a classical optimization problem, which entails finding the minimum cost of transporting homogeneous products from a set of sources (e.g. factories) with their given supplies to meet the given demands at a set of destinations (e.g. warehouses). The production-transportation problem extends the classical transportation problem by introducing a production-related variable at each of the given sources, which decides the supply available at that source. The problem entails finding the production quantity at each source, besides the transportation quantities from the supply sources to meet the demands at the destinations, at the minimum total production and transportation cost. To formally define a production-transportation problem, let $G= (V, U, E)$ be a bipartite graph, where $V$ and $U$ denote the sets of $m$ sources and $n$ destinations, respectively, and $E$ denotes the set of $m \times n$ transportation arcs between the sources and the destinations. Let $c_{ij}$ be the non-negative transportation cost per unit of the product on arc $(i,j) \in E$, and $\phi_i(y_i)$ be the cost of producing $y_i$ units at source $i \in V$. Further, let $d_j$ and $k_i$ represent the demand at destination $j \in U$ and the production capacity at source $i \in V$, respectively. If we define $x_{ij}$ as the amount of the product transported on the arc from $i$ to $j$, and $y_{i}$ as the production quantity at source $i$, then a production-transportation problem can be mathematically stated as:
\begin{align}
	\min_{x,y} & \; \; \sum_{(i,j)\in E}c_{ij}x_{ij} \; + \; \sum_{i\in V} \phi_i(y_i) \label{eq:prod-trans_Obj}\\
	\st & \notag \\
	& \sum_{j\in U} x_{ij} \leq y_i, && \forall \; i \in V \label{eq:prod-trans_supply_cap}\\
	& y_i \leq k_i, && \forall \; i \in V \label{eq:prod-trans_prod_cap}\\
	& \sum_{i\in V} x_{ij} \geq d_j, && \forall \; j \in U \label{eq:prod-trans_demand}\\
	& x_{ij} \geq 0, && \forall \; (i,j) \in E \label{eq:prod-trans_x_nonneg}\\
	& y_i \geq 0, && \forall \; i \in \; V \label{eq:prod-trans_y_nonneg}
\end{align}
 %==============================================
\eqref{eq:prod-trans_Obj}-\eqref{eq:prod-trans_y_nonneg} specifically models the multiple sourcing version of the problem, by allowing any destination $j \in V$ to receive its shipment in parts from several supply sources $i \in U$. The single sourcing variant of the problem, which is also common in the literature, requires that any destination $j \in V$ receives its shipment from only one supply source $i \in U$. This is modelled by imposing a binary restriction on the $x$ variables in the model. The model for single sourcing is provided in \ref{prod_single}. \par
In this paper, we are interested in testing the efficacy of the IA algorithm, as described in Section~\ref{sec:algorithmDescription}, in solving the non-linear production-transportation problem in which the production cost $\phi_i(y_i)$ is concave. To the best of our knowledge, \cite{sharp1970decomposition} were the first to report a non-linear production-transportation problem. However, the production cost $\phi_i(y_i)$ was assumed to be convex, which is relatively easier than its concave counterpart. The production-transportation problem can be viewed as a capacitated minimum cost network flow problem (MCNF) having $(m)$ variables representing the production cost function and $(mn)$ variables representing the transportation cost function. For $m \ll n$, the production-transportation problem with concave production cost has a low-rank concavity \citep{konno1997low}. \cite{guisewite1993polynomial,klinz1993minimum,kuno1997pseudo,kuno1997pseudoa,tuy1993strongly,tuy1993stronglya,tuy1996strongly} have proposed methods specifically suited when the problem has low-rank concavity. These methods belong to a group of polynomial or pseudo-polynomial algorithms in $n$, which do not scale well for $m >3$.
More scalable approaches are B\&B based algorithms, which consist of two varieties. For the single source uncapacitated version of minimum concave cost network flow problem, \cite{gallo1980algorithm,guisewite1991global} implicitly enumerate the spanning tree of the network.\cite{falk1969algorithm,soland1971algorithm,horst1976algorithm,benson1985finite,locatelli2000finite} use linear under-estimators to approximate the concave function, which is improved by partitioning the feasible space. Later, \cite{kuno2000lagrangian} proposed a Lagrangian relaxation-based B$\&$B to solve the multiple sourcing production-transportation problems with concave cost. Subsequently, \cite{saif2016supply} used Lagrangian relaxation-based B$\&$B approaches to solve both the multiple and single sourcing versions of the problem. Recently, \cite{wu2021solving} proposed a deterministic annealing neural network-based method and two neural networks to solve the multiple sourcing version of the production-transportation problem. The authors tested the method for problems with small dimensions. They neither tested the method on high dimensional data-sets nor compared the computational performances against existing methods.\par  
%
% The first type \citep{gallo1980algorithm,guisewite1991global} implicitly enumerates the spanning tree of the network for solving the single source uncapacitated version of the minimum concave cost network flow problem by using the fact that one feasible solution will be optimal, whereas the second type \citep{falk1969algorithm,soland1971algorithm,horst1976algorithm,benson1985finite,locatelli2000finite} approximates the concave function by its linear underestimator and improves the approximation by dividing the feasible space. 
%
%The production-transportation problem with a concave cost is a kind of capacitated minimum cost network flow problem (MCNF) having $(m)$ non-linear variables representing the concave function and $(mn)$ variables representing linear functions. Researchers have reported several algorithms \citep{guisewite1993polynomial,klinz1993minimum,kuno1997pseudo,kuno1997pseudoa,tuy1993strongly,tuy1993stronglya,tuy1996strongly} using this low-rank concavity \citep{konno1997low} to solve the problem with fixed $m$. These algorithms belong to a group of polynomial or pseudo-polynomial algorithms in $n$ and don’t scale well when $m >3$. 
The literature on production-transportation problems has also seen several other variants/extensions of the basic problem. \cite{holmberg1999production} studied a production-transportation problem with concave production cost and convex transportation cost, resulting in a difference of convex (DC) optimization problem, which is solved using a B\&B method. 
%Stochasticity in demand results in convex transportation costs. Thus, the problem is formulated as a difference of convex (DC) optimization problem, which is solved using a B\&B method. 
%Hence, the objective function of the problem can be expressed a difference of convex (DC) functions. 
%The authors formulated the problem as a DC optimization problem and solved it by a B\&B procedure.
%\cite{nagai2005simplicial} reported a simplified B\&B algorithm for solving production-transportation problems with inseparable concave production costs.
\cite{nagai2005simplicial} studied production-transportation problems with inseparable concave production costs, which is solved using a B\&B method. \cite{condotta2013tabu} studied a production scheduling-transportation problem with only one supply source and one destination. The objective of the problem is to schedule the production of a number of jobs with given release dates and processing times, and to schedule their transportation to the customer using a number of vehicles with limited capacity so as to minimize the maximum lateness.\par
%\cite{kang2016enhanced} reported an integrated model for production-transportation problem in a multiple vehicle environment.\par
% %======================================
% \subsection{Production-Transportation Problem with Single Sourcing}
% %======================================
% In the production-transportation problem with single sourcing, $x_{ij}$ is a binary variable, and it retains value $1$ only if the warehouse $j$ is served by the factory $i$. The single-sourcing model is as follows.
% \begin{align*}
% 	\min_{x,y} & \; \; \sum_{(i,j)\in E}c_{ij}d_j x_{ij} \; + \; \sum_{i\in V} \phi_i(y_i) \\
% 	\st & \\
% 	& \sum_{i\in V} x_{ij} = 1, && \forall \; j \in U \\
% 	& y_i \leq k_i, && \forall \; i \in V \\
% 	& y_i= \sum_{j\in U} d_jx_{ij} , && \forall \; j \in U\\
% 	& x_{ij} \in \{0,1\} && \forall \; (i,j) \in E \\
% 	& y_i \geq 0, && \forall \; i \in \; V 
% \end{align*}
% There is only a limited literature on single-sourcing problems. Several real-life problems can be formulated as single sourcing problems, such as, capacitated facility location problem with single sourcing \citep{holmberg1999exact,diabat2016capacitated}, multi-period single-sourcing problem \citep{romeijn2001probabilistic,ahuja2007heuristic,romeijn2003asymptotically,freling2003branch,romeijn2004asymptotic} and transportation problem with single-sourcing \citep{romeijn2011stochastic}. As mentioned earlier, \cite{saif2016supply} solved the production-transportation problem with multiple as well as single-sourcing.
%
Next, we describe our computational experiments on both the multiple and single sourcing versions of the production-transportation problem using our proposed IA algorithm. %as described in Section~\ref{sec:algorithmDescription}.\par %Further, compare the computational performances with the most recent and efficient approaches available in the literature.
%
%Since \cite{saif2016supply} solved the production-transportation problem with multiple as well as single-sourcing using a Lagrangian based B$\&$B approach with Lagrangian heuristics, we use the recent study as a benchmark in our paper. 

%=====================================
\subsection{Computational Experiments}\label{sec:transresults}
%=====================================
In this section, we present the data generation technique, followed by computational results for the multiple sourcing and single sourcing versions of the production-transportation problem. The choice of the solver, platform, and server configuration remains the same as reported in Section~\ref{sec:compuExperKnapsack}.
%Here we present the data generation method, followed by computational results for productions-transportation problems with single and multiple sourcing. 
%All the experiments are performed on a server with dual-core Intel(R) Xenon(R) Gold 6140 CPU @2.3 GHz, 64 GB RAM. %Neither \cite{saif2016supply} nor \cite{kuno2000lagrangian} has reported the configuration of the machine used in their computational experiments. Hence, we report our algorithm's computational times along with their reported computational times. 
%The proposed IA algorithm is coded in C++, and as discussed earlier, the MIP in step 5 of Algorithm~\ref{alg:ConcaveProblem} is solved using the default Branch\&Cut solver of CPLEX 12.7.1. 
The experiments are set to terminate using $\epsilon=0.01$ in step 4 of Algorithm~\ref{alg:ConcaveProblem} or a maximum CPU time limit, whichever reaches earlier. A maximum CPU time of 30 minutes is used for multiple sourcing, and that of 7 hours is used for single sourcing problems. 
%In contrast, the CPU time limit for the single sourcing case is set to 5 hours for the single-sourcing case keeping the $\epsilon$ setting similar to multiple-sourcing.
%====================================
\subsubsection{Data-Set}
%====================================
The data used in the experiments are generated using the scheme described by \cite{kuno2000lagrangian}. The concave cost function, $\phi_i(y_i) = \gamma \sqrt{y_i}$, where $\gamma \sim$ Uniform$\{10, 20\}$; number of sources, $m=|V| \in \{5, \ldots, 25\}$ for multiple sourcing and $m=|V| \in \{5, \ldots, 15\}$ for single sourcing; number of destinations, $n=|U| \in \{25, \ldots, 100\}$; transportation cost, $c_{ij} \sim$ Uniform$\{1, 2, \ldots, 10\} \: \forall \; (i,j) \in E$; production capacity at source $i$, $k_i = 200 \: \forall \; i \in V$; demand at destination $j$, $d_j = \left\lceil\frac{\alpha\sum_{i \in V} k_i}{|U|}\right\rceil \: \forall \; j \in U$, where $\alpha \in \{0.60, 0.75, 0.90\}$ is a measure of capacity tightness.
% $c_{ij} \in \{1, 2, \ldots, 10\} \: \forall (i,j) \in E$. 
% The transportation costs, $c_{ij}, (i,j) \in E$ are integers which are drawn randomly from a discrete distribution, as $c_{ij} \in [1,10]$. Capacities, $k_i$ are fixed at 200, and $d_j$ are kept the same for all $j$ with a rounded value (to nearest integer) of $\alpha (\sum_{i\in V} k_i /|U|)$, where $|U|$ denotes the cardinality of $U$. In our experiments we consider three different values of $\alpha$, i.e. 0.6, 0.75 and 0.9. The concave production cost is assumed to be $f_i(y_i) = \gamma \sqrt{y_i}$, where $\gamma$ is randomly drawn from a discrete distribution, such that, $\gamma \in [10,20]$. The cardinalities of $V$ and $U$, denoted as $m$ and $n$ henceforth, $m \in\{5,\ldots,30\}$ and $n \in \{25,\ldots,100\}$, for production-transportation problem with multiple sourcing, whereas $m \in\{5,\ldots,15\}$ and $n \in \{25,\ldots,100\}$ for production-transportation problem for the single sourcing case. Therefore, the largest problem that we solve, has more than 3000 decision variables. Within every problem size, i.e., fixed $m$ and $n$, 10 random data instances are generated.

%#################################################
\subsubsection{Computational Results}
Tables~\ref{table:tabprodmulti1}-\ref{table:tabprodmulti3} provide a comparison of the computational performance of the IA algorithm against those reported by \cite{kuno2000lagrangian} and \cite{saif2016supply}.
%We present the computational results for the production-transportation problems with multiple sourcing in Tables \ref{table:tabprodmulti1} - \ref{table:tabprodmulti3}, where we compare the computational performance of our proposed method against the methods reported by \cite{kuno2000lagrangian} and \cite{saif2016supply}. 
The columns \cite{kuno2000lagrangian} and \cite{saif2016supply} represent the computational results reported by the respective authors. The missing values in some of the rows indicate that the authors did not provide results for the corresponding data instances. Since both \cite{kuno2000lagrangian} and \cite{saif2016supply} reported only the average and the maximum CPU times over 10 randomly generated test instances (each corresponding to a randomly selected pair of values of $\gamma$ and $c_{ij}$) for each size of the problem, we also do the same for a meaningful comparison. For each problem size, the best average CPU time among the three methods is highlighted in boldface. The following observations can be immediately made from the tables: (i) Of the very selected instances for which \cite{saif2016supply} has reported the computational results, his method never performs the best except for a few very easy instances that can be solved within a fraction of a second. (ii) Between the remaining two methods, our IA algorithm outperforms \cite{kuno2000lagrangian} on the majority of the instances for which the results have been reported by the latter.
% computational time as average and a maximum of 10 random test problems, we represent the computational results accordingly in aggregated forms for meaningful comparisons.  
%From Tables \ref{table:tabprodmulti1} - \ref{table:tabprodmulti3}, it is obvious that we are able to converge to the global optima using the proposed IA algorithm for all the data-sets. 
% Although in terms of computational time, the proposed algorithm performs quite similar to \cite{kuno2000lagrangian} and \cite{saif2016supply} while solving low dimensional data instances, but the proposed algorithm performs way faster than the benchmark algorithms for most of the high dimensional data instances. Specifically, the average computational time for solving the data-sets with dimensions $m=15$ and $n=\{25,50\}$ using the proposed algorithm is 3.47 seconds, which is less than one-twelfth of 44.06 seconds required by \cite{kuno2000lagrangian} for $\alpha=0.6$. 
When $\alpha= 0.75$, for which \cite{kuno2000lagrangian} have reported their results across all the problem sizes used in our experiments (refer to Table~\ref{table:tabprodmulti2}), their method takes 58.49 seconds on average, compared to 5.06 seconds taken by our IA algorithm. To further see the difference between the two methods, we present their performance profiles (created based on the average CPU times) in Figure~\ref{fig:perprofileprod}. The figure shows the IA algorithm to be better on 68.75\% of the instances, while the method by \cite{kuno2000lagrangian} performs better on the remaining 31.25\%. Further, on the instances on which the method by \cite{kuno2000lagrangian} performs worse, it is unable to solve around 50\% of them to optimality even after taking 16 ($=2^4$) times the CPU time taken by the IA algorithm.
% the average computational time for $m=\{15,20,25\}$ and $n=\{25,50,75,100\}$ using our algorithm is 4.99 seconds which is almost 13 times smaller in magnitude than 65.08 seconds taken by \cite{kuno2000lagrangian}; however, \cite{kuno2000lagrangian}'s method outperforms IA algorithm at dimensions $5\times 25,10 \times 25, 25 \times 75, 30 \times 75$ and $30 \times 100$, if we compare the average computational time. This behavior may be attributed to specific data-set issues, since for all the other data-sets, our proposed algorithm is usually performing significantly better. For the case when $\alpha= 0.9$, $m=15$ and $n=\{25,50\}$, once again our proposed algorithm is almost $7.24$ times faster. 
% %Our algorithm takes 0.11 seconds on average to solve the problem for while \cite{kuno2000lagrangian}'s algorithm takes 0.86 sec to solve. 
% %Therfore, from Figures 5-7, it can be inferred that IA algorithm outperforms for most of the data instances. 
% Overall, we are able to converge within 15 minutes of starting the experiments for all the test instances, which highlights the scalability of our algorithm in solving high-dimensional problems.
%whereas \cite{kuno1997pseudo} and \cite{saif2016supply} did not reported computational times for many of the higher dimension. This shows the scalability of our algorithms in converging high-dimensional data-sets.
% Table generated by Excel2LaTeX from sheet 'production transportion compari'
% Table generated by Excel2LaTeX from sheet 'production transportion fin2.O'

\begin{table}[htbp]
  \centering
  \begin{threeparttable}
  \caption{Experimental results of production-transportation problem with multiple sourcing ($\alpha = 0.6$)}
    \begin{tabular}{ccccccccccc}
    \toprule
          &       & \multicolumn{9}{c}{CPU Time (seconds)} \\
\cmidrule{3-11}          &       & \multicolumn{2}{c}{IA Algorithm} &       &       & \multicolumn{2}{c}{\cite{kuno2000lagrangian}} &       & \multicolumn{2}{c}{\cite{saif2016supply}} \\
\cmidrule{3-4}\cmidrule{7-8}\cmidrule{10-11}     \multicolumn{1}{c}{m$\times$n} &       & Avg   & Max   &       &       & Avg   & Max   &       & Avg   & Max \\
    \midrule
    \multicolumn{1}{c}{5$\times$25} &       & 0.52  & 1.44  &       &       & \textbf{0.21} & 0.37  &       & 0.35  & 0.66 \\
    \multicolumn{1}{c}{5$\times$50} &       & 1.09  & 2.99  &       &       & 1.65  & 2.43  &       & \textbf{0.86} & 1.89 \\
    \multicolumn{1}{c}{5$\times$75} &       & 1.03  & 2.24  &       &       & -     & -     &       & -     & - \\
    \multicolumn{1}{c}{5$\times$100} &       & 1.48  & 3.49  &       &       & -     & -     &       & -     & - \\
    \multicolumn{1}{c}{10$\times$25} &       & \textbf{1.06} & 2.59  &       &       & 3.13  & 8.03  &       & 2.43  & 5.41 \\
    \multicolumn{1}{c}{10$\times$50} &       & \textbf{8.37} & 25.29 &       &       & 71.46 & 239.17 &       & 20.43  & 34.48 \\
    \multicolumn{1}{c}{10$\times$75} &       & 13.99 & 81.97 &       &       & -     & -     &       & -     & - \\
    \multicolumn{1}{c}{10$\times$100} &       & 72.92 & 284.43 &       &       & -     & -     &       & -     & - \\
    \multicolumn{1}{c}{15$\times$25} &       & 4.41  & 26.13 &       &       & \textbf{0.44} & 1.25  &       & -     & - \\
    \multicolumn{1}{c}{15$\times$50} &       & \textbf{5.42} & 15.53 &       &       & 87.68 & 260.82 &       & -     & - \\
    \multicolumn{1}{c}{15$\times$75} &       & 156.88 & 733.45 &       &       & -     & -     &       & -     & - \\
    \multicolumn{1}{c}{15$\times$100} &       & 64.81 & 190.73 &       &       & -     & -     &       & -     & - \\
    \multicolumn{1}{c}{20$\times$75} &       & 156.85 & 1247.16 &       &       & -     & -     &       & -     & - \\
    \multicolumn{1}{c}{20$\times$100} &       & 81.49 & 667.31 &       &       & -     & -     &       & -     & - \\
    \multicolumn{1}{c}{25$\times$75} &       & 3.80  & 17.76 &       &       & -     & -     &       & -     & - \\
    \multicolumn{1}{c}{25$\times$100} &       & 6.36  & 32.85 &       &       & -     & -     &       & -     & - \\
    \midrule
    Avg   &       & 36.28 & 208.46 &       &       & -     & -     &       & -     & - \\
\bottomrule
    \end{tabular}%
       \begin{tablenotes}[normal,flushleft]
\item - denotes that the result is not provided by the respective author
\end{tablenotes}
  \label{table:tabprodmulti1}%
  \end{threeparttable}
\end{table}%
% Table generated by Excel2LaTeX from sheet 'production transportion final'
% Table generated by Excel2LaTeX from sheet 'production transportion fin2.O'
\begin{table}[htbp]
  \centering
\begin{threeparttable}
  \caption{Experimental results of production-transportation problem with multiple sourcing ($\alpha = 0.75$)}
    \begin{tabular}{ccccccccccc}
    \toprule
          &       & \multicolumn{9}{c}{CPU Time (seconds)} \\
\cmidrule{3-11}          &       & \multicolumn{2}{c}{IA Algorithm } &       &       & \multicolumn{2}{c}{\cite{kuno2000lagrangian}} &       & \multicolumn{2}{c}{\cite{saif2016supply}} \\
\cmidrule{3-4}\cmidrule{7-8}\cmidrule{10-11}    m$\times$n  &       & Avg   & Max   &       &       & Avg   & Max   &       & Avg   & Max  \\
    \midrule
    5$\times$25 &       & 0.15  & 0.34  &       &       & \textbf{0.08} & 0.18  &       & 0.09  & 0.17 \\
    5$\times$50 &       & \textbf{0.26} & 0.57  &       &       & 1.04  & 1.50  &       & 0.29  & 0.55 \\
    5$\times$75 &       & \textbf{0.30} & 0.54  &       &       & 6.20   & 10.38 &       & -     & - \\
    5$\times$100 &       & \textbf{0.43} & 0.93  &       &       & 19.25 & 30.48 &       & -     & - \\
    10$\times$25 &       & 2.11  & 11.39 &       &       & \textbf{0.30} & 0.78  &       & 0.61   & 3.06 \\
    10$\times$50 &       & \textbf{1.33} & 4.12  &       &       & 6.85  & 11.20  &       & 7.84   & 35.47 \\
    10$\times$75 &       & \textbf{4.40} & 19.52 &       &       & 55.41 & 115.10 &       & -     & - \\
    10$\times$100 &       & \textbf{15.41} & 62.08 &       &       & 334.64 & 1447.67 &       & -     & - \\
    15$\times$25 &       & 0.91  & 3.91  &       &       & \textbf{0.30} & 0.43  &       & -     & - \\
    15$\times$50 &       & \textbf{0.93} & 1.85  &       &       & 8.42  & 16.80  &       & -     & - \\
    15$\times$75 &       & \textbf{2.58} & 6.06  &       &       & 130.84 & 395.43 &       & -     & - \\
    15$\times$100 &       & \textbf{2.72} & 13.52 &       &       & 122.21 & 273.50 &       & -     & - \\
    20$\times$75 &       & 13.28 & 111.66 &       &       & \textbf{11.85} & 17.32 &       & -     & - \\
    20$\times$100 &       & \textbf{15.40} & 92.13 &       &       & 134.98   & 657.88 &       & -     & - \\
    25$\times$75 &       & 19.68 & 52.24 &       &       & \textbf{12.76} & 16.85 &       & -     & - \\
    25$\times$100 &       & \textbf{1.08} & 3.47  &       &       & 90.78 & 175.35 &       & -     & - \\
    \midrule
    Avg   &       & \textbf{5.06} & 24.02 &       &       & 58.49 & 198.18 &       & -     & - \\
    \bottomrule
    \end{tabular}%
     \begin{tablenotes}[normal,flushleft]
\item - denotes that the result is not provided by the respective author
\end{tablenotes}
  \label{table:tabprodmulti2}%
  \end{threeparttable}
\end{table}%

% Table generated by Excel2LaTeX from sheet 'production transportion fin2.O'
\begin{table}[htbp]
  \centering
     \begin{threeparttable}
  \caption{Experimental results of Production-Transportation problem with multiple sourcing ($\alpha = 0.9$)}
    \begin{tabular}{ccccccccccc}
    \toprule
          &       & \multicolumn{9}{c}{CPU Time (seconds)} \\
\cmidrule{3-11}          &       & \multicolumn{2}{c}{IA Algorithm } &       &       & \multicolumn{2}{c}{\cite{kuno2000lagrangian}} &       & \multicolumn{2}{c}{\cite{saif2016supply}} \\
\cmidrule{3-4}\cmidrule{7-8}\cmidrule{10-11}    m$\times$n  &       & Avg   & Max   &       &       & Avg   & Max   &       & Avg   & Max  \\
\cmidrule{1-4}\cmidrule{7-11}    5$\times$25 &       & 0.17  & 0.84  &       &       & 0.04  & 0.05  &       & \textbf{0.03} & 0.08 \\
    5$\times$50 &       & 0.15  & 0.67  &       &       & 0.60  & 1.08  &       & \textbf{0.08} & 0.22 \\
    5$\times$75 &       & 0.10  & 0.27  &       &       & -     & -     &       & -     & - \\
    5$\times$100 &       & 0.21  & 0.62  &       &       & -     & -     &       & -     & - \\
    10$\times$25 &       & 0.17  & 0.84  &       &       & 0.12   & 0.13   &       & \textbf{0.07} & 0.23 \\
    10$\times$50 &       & \textbf{0.20} & 0.51  &       &       & 1.07   & 1.65   &       & 1.26  & 7.27 \\
    10$\times$75 &       & 0.49  & 3.01  &       &       & -     & -     &       & -     & - \\
    10$\times$100 &       & 0.23  & 0.63  &       &       & -     & -     &       & -     & - \\
    15$\times$25 &       & \textbf{0.16} & 0.45  &       &       & 0.24   & 0.28   &       & -     & - \\
    15$\times$50 &       & \textbf{0.18} & 0.31  &       &       & 1.48   & 1.78   &       & -     & - \\
    15$\times$75 &       & 0.25  & 0.51  &       &       & -     & -     &       & -     & - \\
    15$\times$100 &       & 0.27  & 0.48  &       &       & -     & -     &       & -     & - \\
    20$\times$75 &       & 0.17  & 0.46  &       &       & -     & -     &       & -     & - \\
    20$\times$100 &       & 0.12  & 0.21  &       &       & -     & -     &       & -     & - \\
    25$\times$75 &       & 5.52  & 10.15 &       &       & -     & -     &       & -     & - \\
    25$\times$100 &       & 6.25  & 13.64 &       &       & -     & -     &       & -     & - \\
    \midrule
    Avg   &       & 0.91  & 2.10  &       &       & -     & -     &       & -     & - \\
    \bottomrule
    \end{tabular}%
     \begin{tablenotes}[normal,flushleft]
\item - denotes that the result is not provided by the respective author
\end{tablenotes}
  \label{table:tabprodmulti3}%
  \end{threeparttable}
\end{table}%

% \resizebox{\columnwidth}{!}{%
%===========================================
% Since \cite{kuno2000lagrangian} reported the computational time only for very few instances corresponding to $\alpha = 0.6,0.9$, we do not provide a performance profile for these cases. For $\alpha= 0.75$, we compare the computational performance of our proposed algorithm with the performance of \citep{kuno2000lagrangian}'s algorithm. Average computational time was considered for creating the performance profile shown in Figure \ref{fig:perprofileprod}. 
%The "K and U (2000)'s algorithm" indicated in Figure 9 represents the performance profile for \cite{kuno2000lagrangian}'s algorithm. 
% From the figure, it can be inferred that our proposed algorithm solves 72 percent of the data instances faster for $\alpha=0.75$, while \citep{kuno2000lagrangian} is unable to solve more than half of these instances after 16 times the computational time taken by the IA algorithm.
\begin{figure}
		\centering
 	\includegraphics[width=0.6\linewidth]{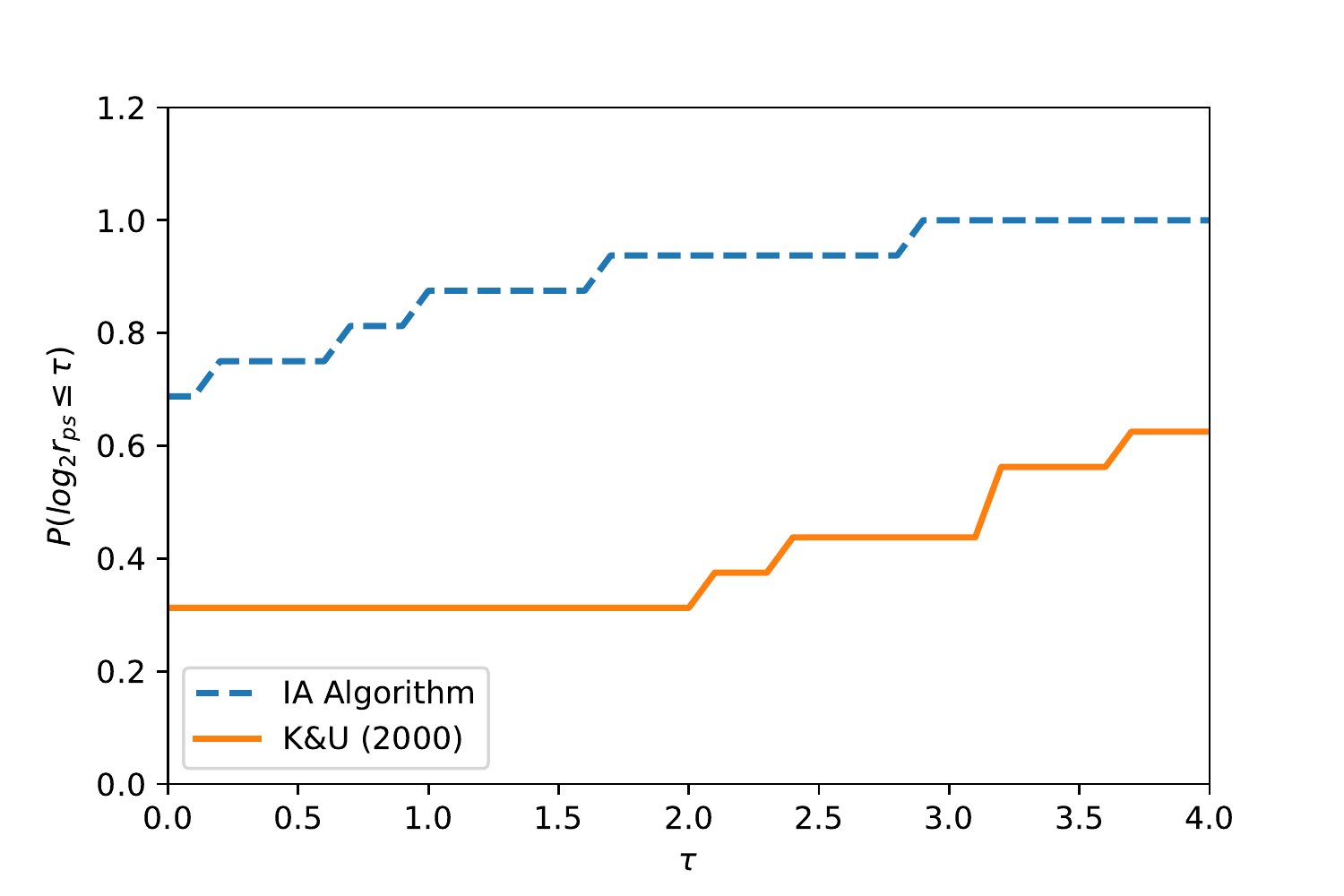}
 	 \caption{Performance profile of production-transportation problem with multiple sourcing for $\alpha=0.75$}
		\label{fig:perprofileprod}
\end{figure}
%================================================
For the production-transportation problem with single sourcing, we provide a comparison of the computational performance of the IA algorithm only with \cite{saif2016supply} since the study by \cite{kuno2000lagrangian} is restricted to only the multiple sourcing version of the problem. The computational results of the two methods for the single sourcing version are reported in Tables~\ref{table:tabprodsingle1}-\ref{table:tabprodsingle3}. For each problem size, the better of the two average CPU times is highlighted in boldface. Once again, like the multiple sourcing case, missing values in some of the rows indicate that \cite{saif2016supply} did not provide results for those data instances. Please note that when the capacity is tight (i.e., $\alpha$ is high), the single sourcing constraints (i.e., $x_{ij} \in \{0, 1\} \: \forall (i, j) \in E$) become increasingly difficult to satisfy as $m=|V|$ starts approaching $n=|U|$. For, this reason, the instances of sizes $m=10$, $n=25$; $m=15$, $n=25$; and $m=15$, $n=50$ became infeasible for $\alpha=0.9$, and the corresponding results are not reported in Table~\ref{table:tabprodsingle3}. Clearly, the IA algorithm outperforms the method by \cite{saif2016supply} by at least one order of magnitude on all the instances for which \cite{saif2016supply} have provided their results. We further test the efficacy of the IA method on even larger instances, the corresponding results are provided in Table~\ref{table:tabprodsingle4}. Some of these problem instances become computationally very difficult to solve, for which we set a maximum CPU time limit of $7$ hours. Clearly, the IA algorithm is able to solve all these instances within less than a 1\% optimality gap within the time limit.\par
\begin{table}[htbp]
  \centering
  \begin{threeparttable}
  \caption{Experimental results of Production-Transportation problem with single sourcing ($\alpha = 0.6$)}
    \begin{tabular}{cllllll}
    \toprule
          &       & \multicolumn{5}{c}{CPU Time (seconds)} \\
\cmidrule{3-7}          &       & \multicolumn{2}{c}{IA Algorithm} &       & \multicolumn{2}{c}{\cite{saif2016supply}} \\
\cmidrule{3-4}\cmidrule{6-7}    m$\times$n  &       & \multicolumn{1}{p{4.215em}}{Avg} & \multicolumn{1}{p{4.215em}}{Max} &       & \multicolumn{1}{p{4.215em}}{Avg } & \multicolumn{1}{p{4.215em}}{Max } \\
\toprule    5$\times$25 &       & \textbf{0.62} & 2.30  &       & 1.79  & 3.73 \\
    5$\times$50 &       & \textbf{1.00} & 2.66  &       & 6.04  & 13.38 \\
    5$\times$75 &       & 2.63  & 8.06  &       & -     & - \\
    5$\times$100 &       & 3.90  & 23.24 &       & -     & - \\
    10$\times$25 &       & \textbf{2.24} & 10.09 &       & 22.05 & 40.78 \\
    10$\times$50 &       & \textbf{39.09} & 133.69 &       & 573.93 & 1710.85 \\
    10$\times$75 &       & 183.42 & 709.86 &       & -     & - \\
    15$\times$25 &       & 23.85 & 112.36 &       & -     & - \\
    15$\times$50 &       & 1341.52 & 5303.48 &       & -     & - \\
    Avg   &       & 177.59 & 700.64 &       & -     & - \\
    \bottomrule
    \end{tabular}%
    \begin{tablenotes}[normal,flushleft]
\item - denotes that the result is not provided by the respective author
\end{tablenotes}
 \label{table:tabprodsingle1}%
\end{threeparttable}
\end{table}%
% Table generated by Excel2LaTeX from sheet 'prod-trans single final 2.O'
\begin{table}[htbp]
  \centering
  \begin{threeparttable}
  \caption{Experimental results of Production-Transportation problem with single sourcing ($\alpha = 0.75$)}
    \begin{tabular}{cclllll}
    \toprule
          &       & \multicolumn{5}{c}{CPU Time (seconds)} \\
\cmidrule{3-7}          &       & \multicolumn{2}{c}{IA Algorithm} &       & \multicolumn{2}{c}{\cite{saif2016supply}} \\
\cmidrule{3-4}\cmidrule{6-7}    m$\times$n  &       & \multicolumn{1}{p{4.215em}}{Avg} & \multicolumn{1}{p{4.215em}}{Max} &       & \multicolumn{1}{p{4.215em}}{Avg } & \multicolumn{1}{p{4.215em}}{Max } \\
\toprule   5$\times$25 &       & \textbf{0.23} & 0.52  &       & 1.44  & 2.59 \\
    5$\times$50 &       & \textbf{0.42} & 0.89  &       & 4.17  & 6.65 \\
    5$\times$75 &       & 0.73  & 1.10  &       & -     & - \\
    5$\times$100 &       & 11.75 & 57.04 &       & -     & - \\
    10$\times$25 &       & \textbf{1.21} & 7.08  &       & 27.92 & 53.88 \\
    10$\times$50 &       & \textbf{52.09} & 273.19 &       & 455.51 & 1495.77 \\
    10$\times$75 &       & 51.91 & 272.93 &       & -     & - \\
    15$\times$25 &       & 20.26 & 191.08 &       & -     & - \\
    15$\times$50 &       & 1009.3 & 6672.76 &       & -     & - \\
    \midrule
    Avg   &       & 127.54 & 830.73 &       & -     & - \\
    \bottomrule
    \end{tabular}%
\begin{tablenotes}[normal,flushleft]
\item - denotes that the result is not provided by the respective author
\end{tablenotes}
 \label{table:tabprodsingle2}%
\end{threeparttable}
\end{table}%

% Table generated by Excel2LaTeX from sheet 'prod-trans single final 2.O'
\begin{table}[htbp]
  \centering
  \begin{threeparttable}
  \caption{Experimental results of Production-Transportation problem with single sourcing ($\alpha = 0.9$)}
    \begin{tabular}{cclllll}
    \toprule
          &       & \multicolumn{5}{c}{CPU Time (seconds)} \\
\cmidrule{3-7}          &       & \multicolumn{2}{c}{IA Algorithm} &       & \multicolumn{2}{c}{\cite{saif2016supply}} \\
\cmidrule{3-4}\cmidrule{6-7}    m$\times$n  &       & \multicolumn{1}{p{4.215em}}{Avg} & \multicolumn{1}{p{4.215em}}{Max} &       & \multicolumn{1}{p{4.215em}}{Avg } & \multicolumn{1}{p{4.215em}}{Max } \\
\toprule    5$\times$25 &       & \textbf{0.08} & 0.18  &       & 0.35  & 0.45 \\
    5$\times$50 &       & \textbf{0.48} & 2.02  &       & 1.04  & 2.39 \\
    5$\times$75 &       & 0.68  & 4.53  &       & -     & - \\
    5$\times$100 &       & 2.61  & 10.72 &       & -     & - \\
    10$\times$50 &       & \textbf{0.10} & 0.17  &       & 23.37 & 24.59 \\
    10$\times$75 &       & 62.76 & 356.63 &       & -     & - \\
    15$\times$75 &       & 0.15  & 0.19  &       & -     & - \\
    \midrule
    Avg   &       & 9.55  & 53.49 &       & -     & - \\
    \bottomrule
    \end{tabular}%
     \begin{tablenotes}[normal,flushleft]
\item - denotes that the result is not provided by the respective author
\end{tablenotes}
  \label{table:tabprodsingle3}%
\end{threeparttable}
\end{table}%
% Table generated by Excel2LaTeX from sheet 'prod_trans single high dim'
\begin{table}[htbp]
  \centering
  \caption{Experimental results of Production-Transportation problem with single sourcing}
    \begin{tabular}{cccccccccc}
    \toprule
          &       & \multicolumn{4}{c}{Optimality Gap ($\%$)}       &       & \multicolumn{3}{c}{CPU Time (seconds)} \\
\cmidrule{3-6}\cmidrule{8-10}          &       & \multicolumn{8}{c}{m$\times$n} \\
\cmidrule{3-10}          &       & 10$\times$100 &       & 15$\times$75 & 15$\times$100 &       & 10$\times$100 & 15$\times$75 & 15$\times$100 \\
    \midrule
    \multicolumn{10}{c}{$\alpha= 0.6$} \\
    Avg   &       & 0.01  &       & 0.05  & 0.06  &       & 13570.38 & 7532.62 & 15327.96 \\
    Min   &       & 0.00  &       & 0.00  & 0     &       & 853.90 & 25.06 & 283.89 \\
    Max   &       & 0.05  &       & 0.40  & 0.27  &       & 25200.00 & 25200.00 & 25200.00 \\
          &       &       &       &       &       &       &       &       &  \\
    \multicolumn{10}{c}{$\alpha=0.75$} \\
    Avg   &       & 0.00  &       & 0.06  & 0.09  &       & 5800.01 & 21335.06 & 21056.04 \\
    Min   &       & 0.00  &       & 0.00  & 0     &       & 2.37 & 3.82 & 67.66 \\
    Max   &       & 0.02  &       & 0.25  & 0.25  &       & 25200.00 & 25200.00 & 25200.00 \\
          &       &       &       &       &       &       &       &       &  \\
    \multicolumn{10}{c}{$\alpha=0.9$} \\
    Avg   &       & 0.02  &       & 0.00  & 0.01  &       & 10146.33 & 0.14 & 10413.67 \\
    Min   &       & 0     &       & 0.00  & 0.00  &       & 2.14 & 0.13 & 1.57 \\
    Max   &       & 0.09  &       & 0.00  & 0.03  &       & 25200.00 & 0.19 & 25200.00 \\
    \bottomrule
    \end{tabular}%
  \label{table:tabprodsingle4}%
\end{table}%

\section{Conclusions}\label{sec:conclusions}
In this paper, we proposed an exact algorithm for solving concave minimization problems using a piecewise-linear inner-approximation of the concave function. The inner-approximation of the concave function results in a bilevel program, which is solved using a KKT-based approach. We make theoretical contributions by identifying a tight value of BigM for general problems that can help in efficiently solving the bilevel program that provides a lower bound to the original problem. 
%represents a relaxation of the original formulation. The relaxed problem does not contain concavity and can be solved using any convex solver. At each iteration, a new point is added that results in an improved approximation of the concave function and a tighter relaxed problem. 
Our proposed algorithm guarantees improvement in the lower bound at each iteration and terminates at the global optimal solution. The algorithm has also been tested on two common application problems, namely, the concave knapsack problem and the production-transportation problem. Our extensive computational results show that our algorithm is able to significantly outperform the specialized methods that were reported in the literature for these two classes of problems. We believe that the algorithm will be useful for exactly solving a large number of other concave minimization applications for which practitioners often have to resort to customized methods or heuristics for solving the problem.

%The working of the algorithm has been shown through an example problem and it has also been tested by solving two common application problems, namely, concave knapsack problem and production-transportation problem. Our extensive computational results show that our algorithm is able to outperform the specialized heuristics and exact methods that were reported in the literature for these two classes of problems.

%The algorithm assimilates a new cut strategy of approximating concave function with the addition of inner cuts through an iterative procedure. In each iteration, the concave function is replaced with the addition of inner cuts, and as a result of which, the nonlinear problem gets converted to a MILP, which is solved using standard branch and cut algorithm. Since in each iteration, new points are added, that results in a better approximation of the concave function. Thus in every iteration, lower bound gets improved, thereby decreasing the gap. The upper bound in each iteration is calculated by putting the solution into the objective function of the original problem. 

\bibliography{mybib} 

\newpage

\appendix
\section{Illustrative Example for Concavity in Objective Function}\label{sec:NumEx}
%###########################################################
To illustrate the algorithm, we consider a small-size numerical example:
\begin{align}
	\min_{x} & \; \; \phi(x) = -5x_1^{\frac{3}{2}} + 8x_1 - 30 x_2 \label{eq:ex1obj}\\
	\st & \;\;\; -9x_1 + 5x_2 \leq 9 \\
	& \;\;\; x_1 - 6x_2 \leq 6 \\
	& \;\;\; 3x_1 + x_2 \leq 9 \\
	& x \in X = \{x_j \in \mathbb{Z}^n |1 \leq x_j \leq 7, j = 1,2 \} \label{eq:ex1lc}
\end{align}
Iteration 1: We replace the concave function $-x_1^\frac{3}{2}$ by a new variable $t_1$.
\begin{align*}
	\min_{x} & \; \; \phi(x) = 5t_1 + 8x_1 - 30 x_2 \\
	\st & \;\;\; -9x_1 + 5x_2 \leq 9 \\
	& \;\;\; x_1 - 6x_2 \leq 6 \\
	& \;\;\; 3x_1 + x_2 \leq 9 \\
	& \;\;\; x \in X = \{x_j \in \mathbb{Z}^n |1 \leq x_j \leq 7, j = 1,2 \}\\
	& \;\;\; t_1 \ge -x_1^{\frac{3}{2}}
\end{align*}
Next, we replace the concave constraints with inner-approximation generated using two points, $x_1 \in \{1,7\}$, which gives us the relaxation of the problem \eqref{eq:ex1obj}-\eqref{eq:ex1lc} as bilevel program.
Let $g(x_1) = -x_1^\frac{3}{2}$, then $g(1) = -1$, $g(7) = -18.52$.
\begin{align}
\min_{x} & \; \; \phi(x) = 5t_1 + 8x_1 - 30 x_2 \label{eq:exobjknap1}\\
	\st & \;\;\; -9x_1 + 5x_2 \leq 9 \label{eq:exconsknap1itr1}\\
	& \;\;\; x_1 - 6x_2 \leq 6 \label{eq:exconsknap2itr1}\\
	& \;\;\; 3x_1 + x_2 \leq 9 \label{eq:exconsknap3itr1}\\
	& \;\;\; x \in X = \{x_j \in \mathbb{Z}^n |1 \leq x_j \leq 7, j = 1,2 \} \label{eq:exconsknap4itr1}\\
	& \;\;\; \mu \in \argmax_{\mu} \left\{ -\mu_1 - 18.52 \mu_2 : \mu_1 + \mu_2 = 1, \mu_1 + 7\mu_2 = x_1,-\mu_1 \leq 0 -\mu_2 \leq 0 \right\} \label{eq:exconsknap5itr1}\\
	& \;\;\; t_1 \ge -\mu_1 - 18.52 \mu_2 \label{eq:exconsknap6itr1}
\end{align}
Let $\gamma_1,\gamma_2,\gamma_3, \textrm{and}, \gamma_4$ be the Lagrange multipliers for the constraints in \eqref{eq:exconsknap5itr1}, then the KKT conditions for the lower level program in \eqref{eq:exconsknap5itr1} can be written as follows:
\begin{align}
&1 + \gamma_1 + \gamma_2 - \gamma_3 = 0 \label{eq:exconsknap7itr1}\\
&18.52 + \gamma_1 + 7 \gamma_2 - \gamma_4 = 0 \label{eq:exconsknap8itr1}\\
&-\mu_1 \gamma_3 = 0 \label{eq:exconsknap9itr1}\\
&-\mu_2 \gamma_4 = 0 \label{eq:exconsknap10itr1}\\
&\mu_1, \mu_2 , \gamma_3, \gamma_4 \geq 0 \label{eq:exconsknap11itr1}\\ 
&\gamma_1, \gamma_2 - \textrm{unrestricted} \label{eq:exconsknap12itr1}
\end{align}
We linearize equations \eqref{eq:exconsknap9itr1} and \eqref{eq:exconsknap10itr1} using the BigM values proposed in Theorem~\ref{th:bigM}. 
\begin{align}
&\mu_1 \leq M Z_1 \label{eq:exconsknap13itr1}\\
&\gamma_3 \leq M (1 - Z_1) \label{eq:exconsknap14itr1}\\
&\mu_2 \leq M Z_2 \label{eq:exconsknap15itr1}\\
&\gamma_4 \leq M (1 - Z_2) \label{eq:exconsknap16itr1}\\
&Z_1 , Z_2 \in \{0,1\} \label{eq:exconsknap17itr1}
\end{align}
The relaxed model for the original problem (\eqref{eq:ex1obj}-\eqref{eq:ex1lc}) is given below as a mixed integer linear program (MILP). 

\begin{align*}
[EX1-1] \min & 5 t_1 + 8 x_1 - 30 x_2 \\
\st &t_1 \geq -\mu_1 - 18.52 \mu_2 \\
& \;\;\; \mu_1 + 7\mu_2 = x_1 \\
%	& \;\;\;-\mu_1 \leq 0 \\
%	& \;\;\;-\mu_2 \leq 0 \\
	& \;\;\;\mu_1 + \mu_2 = 1 \\
 &\;\;\;\eqref{eq:exconsknap1itr1} - \eqref{eq:exconsknap4itr1}, \eqref{eq:exconsknap7itr1} - \eqref{eq:exconsknap8itr1}, \eqref{eq:exconsknap11itr1} - \eqref{eq:exconsknap17itr1} 
\end{align*}
The above formulation can be solved using an MILP solver to arrive at the following solution, $x_1 = 2, x_2 = 3, \textrm{ objective value } = -93.6$.
Hence, the lower bound is -93.6 and the upper bound is -88.14.\\
Iteration 2: The solution obtained from iteration 1 gives an additional point, $x_1 = 2$, to approximate $g(x_1)= -x_1^\frac{3}{2}$, where $g(2) = -2.83$. The updated problem with an additional point is given as follows:
\begin{align}
\min_{x} & \; \; \phi(x) = 5t_1 + 8x_1 - 30 x_2 \label{eq:exobjknap2}\\
	\st & \;\;\; \eqref{eq:exconsknap1itr1} - \eqref{eq:exconsknap4itr1} \\
	& \;\;\; \mu \in \argmax_{\mu} \bigg\{ -\mu_1 - 18.52 \mu_2 - 2.83 \mu_3\ : \\ 
	& \hspace{25mm} \sum_{i=1}^{3}\mu_i = 1, \mu_1 + 7\mu_2 + 2 \mu_3= x_1,-\mu_i \leq 0 \; \forall \; i=1,2,3 \bigg\} \label{eq:exconsknap1itr2}\\
	& \;\;\; t_1 \ge -\mu_1 - 18.52 \mu_2 - 2.83 \mu_3 \label{eq:exconsknap2itr2}
\end{align}
Let $\gamma_1,\gamma_2,\gamma_3, \gamma_4, \textrm{and}, \gamma_5$ be the Lagrange Multipliers for the constraints in \eqref{eq:exconsknap1itr2}, the the following represents the KKT conditions for \eqref{eq:exconsknap1itr2}.
\begin{align}
&1 + \gamma_1 + \gamma_2 - \gamma_3 = 0 \label{eq:exconsknap3itr2}\\
&18.52 + \gamma_1 + 7 \gamma_2 - \gamma_4 = 0 \label{eq:exconsknap4itr2}\\
&2.83 + \gamma_1 + 2 \gamma_2 - \gamma_5 = 0 \label{eq:exconsknap5itr2}\\
&-\mu_1 \gamma_3 = 0 \label{eq:exconsknap6itr2}\\
&-\mu_2 \gamma_4 = 0 \label{eq:exconsknap7itr2}\\
&-\mu_3 \gamma_5 = 0 \label{eq:exconsknap8itr2}\\
&\mu_1, \mu_2 , \gamma_3, \gamma_4, \gamma_5 \geq 0 \label{eq:exconsknap9itr2}\\
&\gamma_1, \gamma_2 - \textrm{unrestricted} \label{eq:exconsknap10itr2}
\end{align}
We once again linearize equation \eqref{eq:exconsknap6itr2}- \eqref{eq:exconsknap8itr2} using the BigM values proposed in Theorem~\ref{th:bigM}. 
\begin{align}
&\mu_1 \leq M Z_1 \label{eq:exconsknap11itr2}\\
&\gamma_3 \leq M (1 - Z_1) \label{eq:exconsknap12itr2}\\
&\mu_2 \leq M Z_2 \label{eq:exconsknap13itr2}\\
&\gamma_4 \leq M (1 - Z_2) \label{eq:exconsknap14itr2}\\
&\mu_3 \leq M Z_3 \label{eq:exconsknap15itr2}\\
&\gamma_5 \leq M (1 - Z_3) \label{eq:exconsknap16itr2}\\
&Z_1 , Z_2, Z_3 \in \{0,1\} \label{eq:exconsknap17itr2}
\end{align}
A tighter relaxed problem for \eqref{eq:ex1obj}-\eqref{eq:ex1lc} as compared to the one in iteration 1 is given as follows:
\begin{align*}
[EX1-2] \min & 5 t_1 + 8 x_1 - 30 x_2 \\
\st &t_1 \geq -\mu_1 - 18.52 \mu_2 - 2.83 \mu_3\\
& \mu_1 + \mu_2 + \mu_3= 1 \\
	& \;\;\; \mu_1 + 7\mu_2 + 2 \mu_3= x_1 \\
%	& \;\;\;-\mu_1 \leq 0 \\
%	& \;\;\;-\mu_2 \leq 0 \\
%	& \;\;\; -\mu_3 \leq 0 \\
 &\;\;\; \eqref{eq:exconsknap1itr1} - \eqref{eq:exconsknap4itr1}, \eqref{eq:exconsknap3itr2} - \eqref{eq:exconsknap5itr2}, \eqref{eq:exconsknap9itr2} - \eqref{eq:exconsknap17itr2} 
\end{align*} 
Solution of the above formulation is $x_1 = 2, x_2 = 3, \textrm{ objective value } = -88.15$. The lower bound is -88.15 and the upper bound is -88.14. Additional iterations would lead to further tightening of the bounds. 
%In the next section, we will be discussing two application areas where concavity arises due to economies of scale, and we will report each problem's data generation method followed by their computational results.

\section{Illustrative Example for Concavity in Constraints}\label{sec:constraints}
The proposed algorithm can also solve the class of problems in which concavity is present in the constraints. We illustrate this using an example problem that has been taken from \cite{floudas1999handbook} (refer to Section 12.2.2 in the handbook).  However, for problems with concavity in constraints we have not been able to propose a tight value for BigM.%see above
\begin{align}
	\min_{x,y} & \; \; -0.7y + 5(x_1 - 0.5)^2 + 0.8
	\label{eq:CCexobj}\\
	\st & \;\;\; x_2 \geq -e^{(x_1 - 0.2)} \label{eq:CCexcons1}\\
	& \;\;\; x_2 - 1.1y \leq -1 \label{eq:CCexcons2}\\
	& \;\;\; x_1 - 1.2y \leq 0.2 \label{eq:CCexcons3}\\
	& \;\;\; 0.2 \leq x_1 \leq 1 \label{eq:CCexcons4}\\
	& \;\;\; -2.22554 \leq x_2 \leq -1 \label{eq:CCexcons5}\\
	& \;\;\; y \in \{0,1\} \label{eq:CCexcons6}
\end{align}
The above problem has a convex objective function, but it is nonconvex because of equation \eqref{eq:CCexcons1}. Let us start the iterations with two points, $x_1 \in \{0.2,1\}$. 
Let $h(x_1) = -e^{(x_1 - 0.2)}$, then $h(0.2) = -1, h(1) = -2.22$. Next, we reformulate the problem \eqref{eq:CCexobj}-\eqref{eq:CCexcons6} by replacing the concave constraint with its inner-approximation generated using two points.
\begin{align}
\min_{x} & \; \; \phi(x) = -0.7y + 5(x_1 - 0.5)^2 + 0.8 \\
	\st & \;\;\; \eqref{eq:CCexcons2} - \eqref{eq:CCexcons6}\\
	& \;\;\; \mu \in \argmax_{\mu} \left\{ -\mu_1 - 2.22 \mu_2 : \mu_1 + \mu_2 = 1, 0.2 \mu_1 + \mu_2 = x_1,-\mu_1 \leq 0 -\mu_2 \leq 0 \right\} \label{eq:CCexcons7}\\
	& \;\;\; x_2 \ge -\mu_1 - 2.22 \mu_2 \label{eq:CCexcons8}
\end{align}
Let $\lambda_1,\lambda_2,\lambda_3, \textrm{and}, \lambda_4$ be the Lagrange multipliers of the constraints in \eqref{eq:CCexcons7} then KKT conditions for \eqref{eq:CCexcons7} can be written as:
\begin{align}
&1 + \lambda_1 + 0.2\lambda_2 -\lambda_3 = 0 \label{eq:CCexcons11itr1}\\
&2.22 + \lambda_1 + \lambda_2 -\lambda_4= 0 \label{eq:CCexcons12itr1}\\
&-\mu_1 \lambda_3 = 0 \label{eq:CCexcons13itr1}\\
&-\mu_2 \lambda_4 = 0 \label{eq:CCexcons14itr1}\\
&\mu_1, \mu_2 , \lambda_3, \lambda_4 \geq 0 \label{eq:CCexcons15itr1}\\
&\lambda_1, \lambda_2 - \textrm{unrestricted} \label{eq:CCexcons16itr1}
\end{align}
We linearize equations \eqref{eq:CCexcons13itr1} and \eqref{eq:CCexcons14itr1} using a BigM value.
\begin{align}
\mu_1 &\leq M Z_1 \label{eq:CCexcons17itr1}\\
\lambda_3 &\leq M (1 - Z_1) \label{eq:CCexcons18itr1}\\
\mu_2 &\leq M Z_2 \label{eq:CCexcons19itr1}\\
\lambda_4 &\leq M (1 - Z_2) \label{eq:CCexcons20itr1}\\
Z_1 , Z_2 &\in \{0,1\} \label{eq:CCexcons21itr1}
\end{align}
At iteration 1 we solve the following quadratic program: 
\begin{align*}
[EX2-1] \min & -0.7y + 5(x_1 - 0.5)^2 + 0.8 \\
\st & x_2 \geq -\mu_1 - 2.22 \mu_2 \\
& \;\;\; 0.2 \mu_1 + \mu_2 = x_1 \\
	& \mu_1 + \mu_2 = 1 \\
 &\eqref{eq:CCexcons2} - \eqref{eq:CCexcons6}, \eqref{eq:CCexcons11itr1} - \eqref{eq:CCexcons12itr1}, \eqref{eq:CCexcons15itr1} - \eqref{eq:CCexcons21itr1} 
\end{align*}
The solution of the $EX2-1$ is given as, $x_1 = 0.921, x_2 = -2.1, y = 1, \textrm{ objective value = 0.9875}$. The above solution gives an additional point $x_1 = 0.921$ to approximate the $h(x_1)= e^{(x_1 - 0.2)}$, where $h(0.921) = 2.06$. Hence the updated problem is as follows:
\begin{align}
\min_{x} & \; \; \phi(x) = -0.7y + 5(x_1 - 0.5)^2 + 0.8 \\
	\st & \;\;\; \eqref{eq:CCexcons2} - \eqref{eq:CCexcons6} \\
	& \mu \in \argmax_{\mu} \bigg\{ -\mu_1 - 2.22 \mu_2 - 2.06\mu_3 :\\
	& \hspace{25mm}\sum_{i=1}^{3}\mu_i = 1, 0.2\mu_1 + \mu_2 + 0.921 \mu_3= x_1,-\mu_i \leq 0 \; \forall \; i=1,2,3 \bigg\} \label{eq:CCexcons11itr2}\\
	& x_2 \ge -\mu_1 - 2.22 \mu_2 - 2.06\mu_3 \label{eq:CCexcons11itr3}
\end{align}
Let $\lambda_1,\lambda_2,\lambda_3, \lambda_4, \textrm{and}, \gamma_5$ be Lagrange multipliers for the constraints of equation \eqref{eq:CCexcons11itr2} then the corresponding KKT conditions are as follows:
\begin{align}
&1 + \lambda_1 + 0.2 \lambda_2 - \lambda_3 = 0 \label{eq:CCexcons12itr2}\\
&2.22 + \lambda_1 + \lambda_2 - \lambda_4 = 0 \label{eq:CCexcons13itr2}\\
&2.06 + \lambda_1 + 0.921 \lambda_2 - \lambda_5 = 0 \label{eq:CCexcons14itr2}\\
&-\mu_1 \lambda_3 = 0 \label{eq:CCexcons15itr2}\\
&-\mu_2 \lambda_4 = 0 \label{eq:CCexcons16itr2}\\
&-\mu_3 \lambda_5 = 0 \label{eq:CCexcons17itr2}\\
&\mu_1, \mu_2 , \lambda_3, \lambda_4, \lambda_5 \geq 0 \label{eq:CCexcons18itr2}\\
&\lambda_1, \lambda_2 - \textrm{unrestricted} \label{eq:CCexcons19itr2}
\end{align}
Upon linearization of \eqref{eq:CCexcons15itr2}-\eqref{eq:CCexcons17itr2} using a BigM value we get: 
\begin{align}
&\mu_1 \leq M Z_1 \label{eq:CCexcons20itr2}\\
&\lambda_3 \leq M (1 - Z_1) \label{eq:CCexcons21itr2}\\
&\mu_2 \leq M Z_2 \label{eq:CCexcons22itr2}\\
&\lambda_4 \leq M (1 - Z_2) \label{eq:CCexcons23itr2}\\
&\mu_3 \leq M Z_3 \label{eq:CCexcons24itr2}\\
&\lambda_5 \leq M (1 - Z_3) \label{eq:CCexcons25itr2}\\
&Z_1 , Z_2, Z_3 \in \{0,1\} \label{eq:CCexcons26itr2}
\end{align}
At iteration 2 we solve the following quadratic program:
\begin{align*}
[EX2-2]\min & -0.7y + 5(x_1 - 0.5)^2 + 0.8 \\
\st & \;\;\; x_2 \geq -\mu_1 - 2.22 \mu_2 -2.06 \mu_3 \\
&\;\;\; \mu_1 + \mu_2 + \mu_3= 1 \\
	& \;\;\; 0.2\mu_1 + \mu_2 + 0.921 \mu_3= x_1 \\
 &\;\;\; \eqref{eq:CCexcons2} - \eqref{eq:CCexcons6}, \eqref{eq:CCexcons12itr2} - \eqref{eq:CCexcons14itr2}, \eqref{eq:CCexcons18itr2} - \eqref{eq:CCexcons26itr2} 
\end{align*}
The solution of $EX2-2$ is $x_1 = 0.9419, x_2 = -2.1, y = 1, \textrm{ objective value } = 1.0769$.

The new point $x_1 = 0.9419$ is used in iteration 3, where the solution is $x_1 = 0.9419, x_2 = -2.1, y = 1$ and the lower bound is $1.0765$. The algorithm can be terminated when the violation for the concave constraint is small. In this case, we stop further iterations of the algorithm. The known global optimal solution for the problem is
$x_1 = 0.9419, x_2 = -2.1, y = 1$ with an optimal objective value of $1.0765$ \citep{floudas1999handbook}.
\section{IA Algorithm Versus Gurobi Solver for Concave Quadratic Knapsack Problem} \label{gurobi_comparision}
Non-convex quadratic programs can be solved to optimality by some commercial solvers, like Gurobi. The solvers exploit the quadratic terms in the formulation and convert the non-convex quadratic program into a bilinear program. The bilinear terms can then be handled using envelopes, like McCormick envelopes, in a spatial branching framework. However, this idea cannot be extended to optimization problems with general concave functions. Table~\ref{tab:quadgurabi} provides a comparison of the computational performance of the proposed IA algorithm against Gurobi on the concave quadratic test case. Note that the computational experiments reported in the table have been carried out on the same PC as reported in Section \ref{sec:compuExperKnapsack}. The actual computational times have been reported for both the approaches. Clearly, Gurobi is computationally more efficient than our proposed IA method as it exploits the quadratic structure of the functions in the problem. However, Gurobi solver cannot handle the other three classes (quartic, cubic, and logrithmic) of concave separable integer knapsack problems or the non-quadratic production-transportation problem discussed in this paper. 

%Though Gurobi is computationally more efficient than the proposed method, Gurobi can not solve other classes of concave minimization problems, whereas the IA algorithm is a general algorithm, which can solve all kinds of concave minimization problems. Since several researchers have come up with different classes of efficient algorithms to solve the quadratic concave minimization problems, the focus of this paper is not the quadratic concave minimization problems but the general class of concave minimization problems. 
\begin{table}[htbp]
  \centering
 \caption{Experimental result comparisons for quadratic knapsack problem}
    \begin{tabular}{ccccccccc}
    \toprule
     &  & \multicolumn{7}{c}{CPU Time (seconds)} \\
\cmidrule{3-9}     &  & \multicolumn{3}{c}{IA Algorithm with CPLEX} &  & \multicolumn{3}{c}{Gurobi} \\
\cmidrule{3-5}\cmidrule{7-9}    {Data set(n$\times$m)} &  & Avg & Min & Max  &  & Avg & Min  & Max \\
\toprule
    30$\times$10 &  & 2.01 &  0.12 & 7.34 &  & \textbf{0.25} & 0.08 & 0.54 \\
    40$\times$10 &  & 1.36 & 0.06 & 3.67 &  & \textbf{0.32} & 0.04 & 0.66 \\
    50$\times$10 &  & 1.43 & 0.29 & 4.57 &  & \textbf{0.35} & 0.10 & 0.93 \\
    80$\times$10 &  & 6.93 & 0.58 & 14.45 &  & \textbf{0.62} & 0.11 & 2.10 \\
    150$\times$10 &  & 4.53 & 0.17 & 16.62 &  & \textbf{0.50} & 0.22 & 0.97 \\
    20$\times$15 &  & 1.91 & 0.17 & 6.33 &  & \textbf{0.28} & 0.12 & 0.57 \\
    30$\times$15 &  & 16.01 & 0.29 & 47.35 &  & \textbf{0.70} & 0.27 & 1.69 \\
    40$\times$15 &  & 31.53 & 0.64 & 86.52 &  & \textbf{1.75} & 0.16 & 7.68 \\
    \bottomrule
    Avg &  & 8.21 & 0.29 & 23.36 &  & \textbf{0.60} & 0.14 & 1.89 \\
    \bottomrule
    \end{tabular}%
  \label{tab:quadgurabi}%
\end{table}%

\section{Production-Transportation Problem with Single Sourcing } \label{prod_single}
In multiple sourcing, every destination may accept supply from multiple sources, but the destination can accept supply only from one source in case of single sourcing problem. Hence, in single sourcing binary restriction is imposed in $x_{ij}$ variable. The modified model for production-transportation problem can be stated as follows:
\begin{align}
	\min_{x,y} & \; \; \sum_{(i,j)\in E}c_{ij}x_{ij}d_j \; + \; \sum_{i\in V} \phi_i(y_i) \label{eq:prod-trans_sing_Obj}\\
	\st & \notag \\
	& \sum_{j\in U} x_{ij}d_j \leq y_i, && \forall \; i \in V \label{eq:prod-trans_sing_supply_cap}\\
	& y_i \leq k_i, && \forall \; i \in V \label{eq:prod-trans_sing_prod_cap}\\
	& \sum_{i\in V} x_{ij} \geq 1, && \forall \; j \in U \label{eq:prod-trans_sing_demand}\\
	& x_{ij} \in \{0,1\}, && \forall \; (i,j) \in E \label{eq:prod-trans_sing_x_nonneg}\\
	& y_i \geq 0, && \forall \; i \in \; V \label{eq:prod-trans_sing_y_nonneg}
\end{align}
%\end{APPENDIX}
\end{document}